\documentclass[10pt, reqno]{amsart}
\usepackage{amsthm, amsmath, amsfonts, amssymb, enumerate, textcmds, tensor, enumitem, mathdots}
\usepackage{geometry}
\usepackage{pst-node}
\usepackage{tikz-cd} 
\usepackage{appendix}
\usepackage{graphicx,tikz}
\usepackage{nicefrac}
\usepackage{faktor}
\usepackage[colorlinks=true]{hyperref}
\usepackage{mathtools}

\hypersetup{
	colorlinks,
	citecolor=black,
	filecolor=black,
	linkcolor=black,
	urlcolor=black
}
\usepackage{cleveref}

\usepackage{physics}
\usepackage{enumerate}
\usepackage{pinlabel}
\usepackage[colorinlistoftodos]{todonotes}
\newcommand{\R}{\mathbb R}

\newcommand{\PSL}{\mathrm{PSL}}

\newtheorem{thm}{Theorem}[section]
\newtheorem{lem}[thm]{Lemma}

\newtheorem{prop}[thm]{Proposition}
\newtheorem{cor}[thm]{Corollary}

\newtheorem*{thma'}{Theorem A'}

\newtheorem*{thmb'}{Theorem B'}

\DeclareMathOperator{\C}{\mathbb{C}}

\DeclareMathOperator{\CP}{\mathbb{C}\mathbb{P}}

\DeclareMathOperator{\inners}{\langle \cdot, \cdot \rangle}
\DeclareMathOperator{\Hyp}{\mathbb H}

\newcommand{\ccpair}{(c_1, \overline{c_2})}

\newcommand{\hpair}{\vb*h{(c_1, \overline{c_2})}}

\DeclareMathOperator{\phee}{\varphi}
\DeclareMathOperator{\eps}{\varepsilon}

\newcommand{\redd}[1]{\textcolor{red}{#1}}

\newcommand{\putref}{\textcolor{red}{(PUT REF)}}
\newcommand{\inner}[1]{\langle{#1}\rangle}

\newcommand{\AF}{\mathrm{AF}}

\usepackage{physics}
\usepackage{soul}

\usepackage{scalerel}

\usepackage{stmaryrd}
\usepackage{upgreek}

\newcommand{\TSTS}{\mathcal T(S)\times \mathcal T(\overline S)}

\usepackage{mathrsfs}
\usepackage{soul}

\newcommand{\Is}{\mathtt I}
\newcommand{\II}{\mathtt {II}}
\newcommand{\QF}{\mathrm{QF}}

\theoremstyle{definition}

\newtheorem{defn}[thm]{Definition}
\newtheorem{remark}[thm]{Remark}

\makeatletter
\newcommand\longleftrightarrowfill@{%
  \arrowfill@\leftarrow\relbar\rightarrow}
\makeatother

\begin{document}
\title{Quadratic differentials in Teichmüller Theory and complex harmonic maps}

\author[Christian El Emam]{Christian El Emam}
\address{Christian El Emam: University of Torino, Dipartimento di Matematica ``Giuseppe Peano", Via Carlo Alberto, 10, 10123 Torino, Italy.} \email{christian.elemam@unito.it}

\author[Nathaniel Sagman]{Nathaniel Sagman}
\address{Nathaniel Sagman: University of North Carolina at Chapel Hill, 
3250 Phillips Hall $\#3$, Chapel Hill, NC 27599, United States.} \email{nsagman@unc.edu}
\begin{abstract}
We show that Hitchin–Wolf's harmonic map parametrization of Teichmüller space, the Schwarzian holonomy map, and Uhlenbeck’s almost-Fuchsian holonomy map all arise as restrictions of a single locally biholomorphic map into the $\textrm{PSL}(2,\C)$ character variety.
The map admits a right inverse defined on an open subset of quasi-Fuchsian space, encoding harmonic maps together with the Bers embedding. The construction goes through complex harmonic maps to the space of oriented geodesics of $\mathbb{H}^3$.

Along the way, we further develop the theory of complex harmonic maps, with a focus on the natural $\mathbb{C}^*$-action and the relation to complex projective structures. This leads to an independent application: we obtain a new geometric interpretation of certain real twistor lines in the rank-two Deligne–Hitchin moduli space and we recover a theorem of Gothen–Silva.
\end{abstract}
\maketitle

\section{Introduction}

\subsection{A unifying approach}

Let $S$ be a closed oriented surface of genus $\mathtt g\ge 2$. Owing to the uniformization theorem, the Teichmüller space $\mathcal T(S)$ is characterized as either the moduli space of marked complex structures or the moduli space of marked hyperbolic structures on $S$. Moreover, it identifies with a connected component of the space of conjugacy classes of discrete and faithful--henceforth Fuchsian--representations of $\pi_1(S)$ into $\mathrm{PSL}(2,\mathbb R)$; we denote this component by $\mathrm{Fuch}(S)$. 

The interplay between different models of Teichmüller space is one of the central aspects of its study, leading one to interact simultaneously with several subjects, including complex analysis, hyperbolic geometry, dynamics, and PDE. A recurring feature in this circle of ideas is that some relevant geometric objects in one way or another determine holomorphic quadratic differentials. In this paper we consider three such constructions, listed below. Let $c$ be a complex structure on $S$, let $\mathcal{K}_c$ be the canonical bundle of $(S,c)$, so that the space of holomorphic quadratic differentials is $H^0(\mathcal K_c^2)$, and let $\chi^{an}(\pi_1(S), \PSL(2,\C))$ be the smooth locus of the $\textrm{PSL}(2,\C)$ character variety. For the conjugate Riemann surface, we have the corresponding space $H^0(\mathcal{K}_{\overline{c}}^2)$.

\begin{itemize}
    \item Hitchin-Wolf's parametrization $\mathcal W\colon H^0(S,\mathcal K_c^2)\to \mathcal T(S)$, mapping $q\in H^0(\mathcal K_c^2)$ to the isotopy class of the unique hyperbolic structure $h_q$ such that the identity map $id\colon (S,c)\to (S, h_q)$ is harmonic with Hopf differential $q$. See \cite{Hitchin:1986vp}, \cite{W2}.
  \item The holomorphic Bers embedding $\mathcal{B}e\colon\mathcal{T}({S})\to H^0(\mathcal{K}_{\overline c}^2)$, whose image is a bounded domain. Similarly, one has the conjugate map $\overline{\mathcal{B}e}\colon\mathcal{T}(\overline{S})\to H^0(\mathcal{K}_{ c}^2)$. Viewing the inverse of $\overline{\mathcal{B}e}$ as a map to $\mathcal T(S)\times\{[\overline c]\}$, which we then realize as a subset of $\QF(S)$ using Bers' Simultaneous Uniformization Theorem, extends to the Schwarzian parametrization $\mathcal S\colon H^0(\mathcal K_c^2)\to \chi^{an}(\pi_1(S), \PSL(2,\mathbb C))$, which sends $q\in H^0(\mathcal K_c^2)$ to the holonomy of the unique complex projective structure with Schwarzian derivative $q$.
  \item Uhlenbeck's immersion $\mathcal {U}\colon \mathsf{Uh}(S,c)\to \chi^{an}(\pi_1(S), \PSL(2,\mathbb C))$, where $\mathsf{Uh}(S,c)\subset H^0(\mathcal K_c^2)$ is a particular open neighbourhood of $0$, assigning to $q\in \mathsf{Uh}(S,c)$ the holonomy of the unique almost Fuchsian minimal immersion with second fundamental form $Re(q)$; we denote the image by $\AF(S,c)$ (see Section \ref{subsubsec: almost fuchsian} and \cite{Uh}). Although Uhlenbeck's immersion isn't  strictly considered within Teichm{\"u}ller theory, as a standard character in the study of quasi-Fuchsian $3$-manifolds, it falls within the scope of the broader theory.
\end{itemize}

Although all three constructions involve holomorphic quadratic differentials, the 
question of whether they have a precise relationship has not been thoroughly explored. One link between Hitchin-Wolf's parametrization and the Schwarzian parametrization is Gaiotto's conformal limit (see \cite{Gai}, \cite{DFKMMN}, \cite{collier2024conformallimitsparabolicslnchiggs}), although the geometry of this procedure is at first glance a little obscure (more on this below). The main goal of this paper is to show that the three constructions fit into a single, unified picture, made precise in Theorem~\ref{thm: main} below.

\begin{thm}\label{thm: main}
Let $(S,c)$ be a closed Riemann surface of genus $\mathtt g\ge 2$. There exists an open connected subset $\Omega\subset H^0(\mathcal K_c^2)\times H^0(\mathcal K_{\overline c}^2)$ with 
\[
\Omega\supset \Delta \ \ \cup\  \ H^0(\mathcal K_c^2)\times \{0\}\ \ \cup\  \ \{0\}\times H^0(\mathcal K_{\overline c}^2)\]
and invariant under the $\C^*$-action $\zeta\cdot (q_1,\overline{q_2})=(\zeta^2 q_1,\zeta^{-2}\overline{q_2})$, together with a unique local biholomorphism  $$\Phi: \Omega \to \chi^{an}(\pi_1(S),\PSL(2,\C))$$ satisfying the following.  
    \begin{enumerate}
        \item Restricted to $\Delta$, $\Phi(q,\overline{q})=\mathcal {W}(q)$, where $\mathcal W$ is Hitchin-Wolf's parametrization of Teichmuller space, the latter identified with $\mathrm{Fuch}(S)$.
        \item Restricted to $ H^0(\mathcal K_c^2)\times \{0\}$, $\Phi(q,0)=\mathcal S(-\frac 1 2q)$, where $\mathcal S$ is the Schwarzian parametrization, and similar on $\{0\}\times H^0(\mathcal K_{\overline c}^2).$
        \item $\Omega$ contains the image of $\mathsf{Uh}(S,c)$ under the map $q\mapsto (q,-\overline{q})$. On this image, $\Phi(q,-\overline{q})=\mathcal{U}(q)$, where $\mathcal{U}$ is Uhlenbeck's immersion.
    \end{enumerate}
\end{thm}

Note that $\Omega\subset \Delta$ and $\C^*$-invariance imply that $\Omega$ contains $ H^0(\mathcal K_c^2)\times \{0\}$ and $\{0\}\times H^0(\mathcal K_{\overline c}^2).$
About Theorem \ref{thm: main}, it is interesting to notice that $\Delta$ and $\{(q, -\overline q)\}\cap \Omega$ are totally real submanifolds of maximal dimension inside $\Omega$. Hence, as $\mathcal W$ and $\mathcal U$ are real analytic, automatically there exist unique holomorphic maps in neighborhoods of $\Delta$ and of $\{(q, -\overline q)\}\cap \Omega$ satisfying conditions (1) and (3) from the statement. What is surprising is that these two holomorphic maps extend far away to a  larger domain $\Omega$, they coincide over it, and they give item $(2)$ for free.

On the subset of $\Omega$ that lands in the quasi-Fuchsian space $\QF(S)\subset \chi^{an}(\pi_1(S), \PSL(2,\C))$, we prove that we can invert $\Phi$. From this we deduce Corollary \ref{cor: main cor} below, which relates harmonic maps and the Bers embedding. In the paper, using the biholomorphism $\TSTS\to \QF(S)$ coming from Bers' Simultaneous Uniformization Theorem, we view $\{[c]\}\times \mathcal T(\overline S)$ and $\mathcal T(S)\times \{[\overline{c}]\}$ as complex submanifolds of $\QF(S)$.

\begin{cor}
\label{cor: main cor}
There exists an open subset $\mathcal A\subset \mathcal{QF}(S)$ with 
  \[
  \mathcal A\supset \mathrm{Fuch}(S)\ \cup\ \mathcal T(S)\times\{[\overline c]\}\ \cup\ \{[c]\}\times \mathcal T(\overline S) \ \cup \ \mathrm{AF}(S,c),
  \] 
  and a unique holomorphic embedding 
    \[
    \Psi\colon \mathcal A\to H^0(\mathcal K_c^2)\times H^0(\mathcal K_{\overline c}^2). \] 
    such that, in the identification $\TSTS\cong \QF(S)$,
    for all $[c']\in \mathcal T(S)$:
    \begin{enumerate}
        \item $\Psi([c'],[\overline{c}'])=(\mathcal{W}^{-1}([c']),\overline{\mathcal{W}^{-1}([c'])})$,
        \item $\Psi([c'],[\overline c])=(0,-2 \mathcal{B}e([c']))$, i.e., the second component gives a rescaling of the Bers embedding relative to $c$, and similar for points $([ c],[\overline c'])$.
        \item $\Psi(\mathcal U(q))=(q, -\overline q)$, for all $q\in \mathsf {Uh}(S,c)$.
    \end{enumerate}

\end{cor}

We prove Theorem \ref{thm: main} using \emph{equivariant complex harmonic maps} from the universal cover $\widetilde S$ to the space $\mathbb G$ of oriented geodesics of $\mathbb H^3$. Complex harmonic maps into holomorphic Riemannian manifolds, and the associated notion of complex harmonic $G$-bundles (a generalization of harmonic bundles or stable Higgs bundles), were introduced in \cite{ES_Chmaps}. In some sense, complex harmonic maps provide an analytic continuation for the theory of ordinary harmonic maps, and they have a very rich geometry and theory.

Theorem \ref{thm: main} goes through global analysis, similar to \cite[Theorem A]{ElSa} and \cite[Theorem B]{ES_Chmaps}. In fact, the map $\Phi$ can be partially recovered from the map $\mathcal{B}_G$ from \cite[Theorem B]{ES_Chmaps} for $G=\textrm{PSL}(2,\R)^2$. We decided to give the rank $1$ setting its own simpler treatment and moreover to carry out a dedicated study of the geometry. In particular, what we have here, which cannot be gleaned from \cite[Theorem B]{ES_Chmaps}, is: (1) invariance under the full $\C^*$-action, (2) that $\Phi$ is a local biholomorphism, (3) the relation with almost Fuchsian geometry and Uhlenbeck's immersion, and (4) the invertibility on the quasi-Fuchsian image.

Finally, we have completely avoided using Higgs bundle techniques, not drawing on \cite{ES_Chmaps}, and proved Theorem \ref{thm: main} using just rank $1$ holomorphic Riemannian geometry. The motivation for this choice is that Theorem \ref{thm: main} has the flavour of a classical theorem in Teichm{\"u}ller theory, so it should have a compatible proof.
\begin{remark}
   One of the most classical constructions involving quadratic differentials is via Teichm{\"u}ller geodesics. Although Theorem \ref{thm: main} does not make contact with Teichm{\"u}ller geodesics, we note that they can also be realized using harmonic maps (see \cite{Kuwert1996}, \cite{PanWolf2026}).
\end{remark}

\subsection{Complex harmonic maps}
By using concrete mappings in $\mathbb G$, we also use this paper as an opportunity to provide a more geometric understanding of complex harmonic maps. Some of this goes toward Theorem \ref{thm: main} and some of it does not. For instance, in Section \ref{subsec: examples}, we give basic examples of complex harmonic maps, including:
\begin{itemize}
    \item complex harmonic maps arising from osculation maps of $\mathbb{CP}^1$ structures;
    \item complex harmonic maps to complex geodesics, or equivalently, to $(\C,dz^2)$ and $(\C^*, dz^2/z^2)$;
    \item ``exotic" conformal complex harmonic maps from $\C$ to $\mathbb{G}$.
\end{itemize}
One meta-principle worth making explicit is that any real analytic equation or identity in the theory of ordinary harmonic maps should have a complexified version for complex harmonic maps. We use Theorem \ref{thm: main} to give one such example below.

Let $\omega$ be the Weil-Petersson symplectic form on Teichm{\"u}ller space and let $E: \mathcal{T}(S)\to (0,\infty)$ be the Dirichlet energy functional that associates a class of a hyperbolic metric $[h]$ with the energy of the unique harmonic map $(S,c)\to (S,h)$. Let $X$ be the pushforward via $\mathcal{W}$ of the infinitesimal generator of the $S^1$-action on $H^0(S,\mathcal{K}_c^2)$ given by $e^{i\theta}\cdot q=e^{2i\theta} q$, and let $\iota_X$ be the contraction operator on differential forms. It follows from results in \cite{Hitchin:1986vp} (see the the top of Section 7 and Proposition 11.17 in \cite{Hitchin:1986vp}) that 
\begin{equation}\label{eq: energy circle action}
    dE=\iota_X \omega.
\end{equation}
When we complexify, embedding $\mathcal{T}(S)=\textrm{Fuch}(S)$ into $\chi^{an}(\pi_1,\textrm{SL}(2,\C))$, $\omega$ becomes the holomorphic Goldman symplectic form $\omega^{\C}$. On the harmonic map side, it's best to define an extension $E^{\C}$ for $E\circ \mathcal{W}$ on $\Omega$, rather than an extension of $E$ on a subset of $\chi^{an}$. We define $E^{\C}(q_1,\overline{q_2})$ to be the \textit{complex Dirichlet energy} of the complex harmonic map associated with $(q_1,\overline{q}_2)$ (see Section \ref{sec: definition}). The $S^1$-action on $H^0(S,\mathcal{K}_c^2)$ complexifies to the $\C^*$-action on $\Omega$ from Theorem \ref{thm: main}, and we denote the infinitesimal generator by $Z$. In order to extend (\ref{eq: energy circle action}), we have to pull back $\omega^{\C}$ to $\Omega$ using $\Phi$. The corollary below is immediate from Theorem \ref{thm: main} because $E^{\C}$, $Z$, and $\Phi^*\omega^{\C}$ holomorphically extend $E\circ \mathcal{W}$, (the pullback of) $X$, and $\mathcal{W}^*\omega$ respectively.
\begin{cor}\label{cor: complex energy}
    On $\Omega$, $$dE^{\C}=\iota_{Z} (\Phi^*\omega^{\C}).$$
\end{cor}
Note that if we push $E$ and $\omega$ to $H^0(S,\mathcal{K}_c^2)$ via $\mathcal{W}$,
\eqref{eq: energy circle action} becomes $d(E\circ \mathcal{W})_q(r)=(\Phi^*\omega)(2iq,r)$. If $D$ is the push-forward of the Weil-Petersson covariant derivative, differentiating gives $$\textrm{Hess}_q (E\circ \mathcal{W})(r,s)=(D_r(\mathcal{W}^*\omega))_q(2iq,s)+(\mathcal{W}^*\omega)_q(2ir,s).$$
Setting $q=0$, we obtain $$\textrm{Hess}_0(E\circ \mathcal{W})(r,s)=(\mathcal{W}^*\omega)_0(2ir,s),$$ which is exactly Wolf's expression for the Hessian of the energy from \cite[Theorem 5.7]{W2}. By analyticity, the analogous formula holds for the complex Hessian of the complex energy $E^{\C}$ and the holomorphic Goldman symplectic form $\omega^{\C}$.

\subsection{The $\C^*$-action and the Deligne-Hitchin moduli space}\label{sec: intro, twistor}
We close the introduction with an independent theorem, concerning Higgs bundles, which came out of studying the $\C^*$-action on the set $\Omega$ from Theorem \ref{thm: main}.
This $\C^*$-invariance is a manifestation of a natural $\C^*$-action on the space of complex harmonic maps. In Section \ref{subsec: C* action}, we give some geometric interpretation to this action and explain how it leads to families of $\mathbb{CP}^1$-structures on varying Riemann surfaces. Very briefly, we start with a (not necessarily closed) surface $\Sigma$ and a complex harmonic map $f=(f_1,f_2):(\widetilde{\Sigma},c,\overline{c})\to \mathbb{CP}^1\times \mathbb{CP}^1\backslash \Delta$ such that both $f_1$ and $f_2$ are local diffeomorphisms. Acting by $\zeta\in \C^*$ gives another complex harmonic map $f^\zeta=(f_1^\zeta,f_2^{\zeta})$. We prove in Theorem \ref{thm: A and B} (an ingredient toward Theorem \ref{thm: main}) that at least one of $f_1^\zeta$ and $f_2^{\zeta}$ is a local diffeomorphism, and hence determines a complex projective structure. 

We recall that ordinary harmonic maps from a Riemann surface determine Higgs bundles, and that complex harmonic maps in a similar fashion give complex harmonic bundles. The $\C^*$-action becomes an action on the space of complex harmonic bundles that extends the well-studied $S^1$-action on the space of Higgs bundles. As discussed in \cite[Section 5]{ES_Chmaps}, when we take the $\C^*$-orbit of a Higgs bundle (i.e., a complex harmonic bundle in the real locus), we obtain a real twistor line for the Deligne-Hitchin moduli space (as constructed by Deligne and Simpson, see \cite{SimpsonHodgeFiltration}), i.e., a specific type of section of the twistor space of the hyperK{\"a}hler moduli space of Higgs bundles. Any real twistor line is identified with a family of flat connections of the form 
$$D_\zeta=\overline{\partial}_E+(\overline{\partial}_E)^{*_h}+\zeta\phi +\zeta^{-1}\phi^{*_h}, \qquad \zeta\in \C^*,$$
where $(E,\overline{\partial}_E,\phi)$ is a Higgs bundle and $h$ is a harmonic metric. Real twistor lines are intimately related to the hyperK{\"a}hler geometry of Higgs bundle moduli spaces, and have been the subject of numerous recent works (see, for instance \cite{BeckHellerRoeserEnergy}, \cite{BiswasHellerRoeserReal}, \cite{SchulzNilpotent}, and Gaiotto's conjecture \cite{Gai}, \cite{DFKMMN}, \cite{collier2024conformallimitsparabolicslnchiggs}).

Taking all of this together, considering the $\C^*$-action on complex harmonic maps leads to an interpretation of certain twistor lines as parameterizations of families of $\mathbb{CP}^1$-structures. In fact, this interpretation recovers and extends the main result from Gothen-Silva's paper \cite{silva2024conformal}. Going a bit beyond the context of Theorems \ref{thm: main} and \ref{thm: A and B}, we prove the most general result in this direction (for $\textrm{SL}(2,\R)$-Higgs bundles), studying the $\C^*$-orbits of Higgs bundles that correspond to harmonic maps that might not be immersions (as in \cite{silva2024conformal}). We explain the relation with \cite{silva2024conformal} and give more comments on the theorem in Section \ref{sec: proof of twistor line theorem}. Given a Beltrami form $\mu$ of $L^\infty$-norm less than $1$ on a Riemann surface, we let $c(\mu)$ be the complex structure obtained via the measurable Riemann mapping theorem.
\begin{thm}\label{thm: twistor lines}
    Let $(E,\overline{\partial}_E,\phi)$ be a $\textrm{SL}(2,\R)$-Higgs bundle on a Riemann surface $(\Sigma,c)$ with harmonic metric $h$, and let $$D_\zeta=\overline{\partial}_E+(\overline{\partial}_E)^{*_h}+\zeta\phi +\zeta^{-1}\phi^{*_h}, \qquad \zeta\in \C^*$$ be the associated real twistor line. Assume that the equivariant harmonic map $f$ from $\widetilde{\Sigma}$ to $\mathbb{H}^2$ associated with $h$ either is an orientation preserving branched immersion, or parametrizes a geodesic. Let $\mu_f$ be the Beltrami form of $f$, and $$A_f=\{\zeta \in \C^*: |\zeta|^{2} > ||\mu_f||_\infty\}, \qquad  B_f=\{\zeta \in \C^*: |\zeta|^{2}< ||\mu_f||_\infty^{-1}\}.$$ Then, for $\zeta\in A_f$, $D_\zeta$ is a branched oper over $(\Sigma, c(\zeta^2\mu))$, and for $\zeta \in B_f$, $D_\zeta$ is a branched oper over $(\overline{\Sigma}, c(\zeta^{-2}\overline{\mu}))$. The associated maps to spaces of complex projective structures, $$A_f\to \mathcal{CP}(\Sigma), \qquad B_f\to \mathcal{CP}(\overline{\Sigma}),$$ are holomorphic.
\end{thm}
Note that Fuchsian representations correspond exactly to harmonic immersions, and in this case the $\mathbb{CP}^1$ structures are unbranched. The orientation preserving assumption is a normalization that can be removed (one has to replace $\mu$ with $\mu^{-1}$ and one gets the same result). 

The branched immersion condition implies $|\mu_f|<1$ everywhere. If $||\mu_f||_\infty<1$ (so, for instance, if $\Sigma$ is closed), then $$A_f\cup B_f =\C^*,\qquad A_f\cap B_f\neq \emptyset,$$ so Theorem \ref{thm: twistor lines} indeed gives a statement about the whole $\C^*$-orbit. When $f$ parametrizes a geodesic, we have to omit the unit circle: $|\mu|=1$ everywhere and hence $A_f=\mathbb{D}\backslash \{0\}$ and $B_f=\mathbb{C}\backslash\overline{\mathbb{D}}$. In this case, the path $\zeta\mapsto c(\zeta^2\mu)$, $\zeta \in \mathbb{D}$, parametrizes a complex Teichm{\"u}ller disk, and $\zeta\mapsto c(\zeta^{-2}\overline{\mu})$, $|\zeta|>1$, parametrizes a complex Teichm{\"u}ller disk for the oppositely oriented surface.

 \begin{remark}
   A non-real complex harmonic map gives a section of the Deligne-Hitchin moduli space that need not be real. The complex energy $E^{\C}$ of the complex harmonic map (as in Corollary \ref{cor: complex energy}) is exactly the energy of a twistor section as defined in \cite{BeckHellerRoeserEnergy}.
 \end{remark}

\subsection{Outline of the paper}
In Section \ref{sec: preliminaries} we introduce the holomorphic Riemannian space form $\mathbb{G}$ and give the required background on complex metrics. In Section \ref{sec: c harmonic maps} we discuss complex harmonic maps on general surfaces (no compactness, varying complex structures). We establish basic properties and the Bochner formula in Theorem \ref{thm: Bochner equation and harmonicity}, and then we give many examples, discuss the $\C^*$-action, and prove Theorem \ref{thm: A and B} about $\mathbb{CP}^1$-structures. In Section \ref{sec: main theorem} we specialize to a closed surface with a single complex structure, and prove Theorem \ref{thm: main} and Corollary \ref{cor: main cor}. A result on the Goldman form in Section \ref{subsec: holonomy immersion} is of independent interest. Finally, in Section \ref{sec: Higgs bundles}, we briefly discuss complex harmonic maps, and then we prove Theorem \ref{thm: twistor lines}.

\subsection{Acknowledgments}
The first author is funded by the European Union (ERC, \textit{GENERATE}, 101124349). Views and opinions expressed are however those of the author(s) only and do not necessarily reflect those of the European Union or the European Research Council Executive Agency. Neither the European Union nor the granting authority can be held responsible for them. 

The second author gratefully acknowledges that a portion of this research was completed while he was a visiting scholar at the Fields Institute for Research in Mathematical Sciences.

\section{Preliminaries}\label{sec: preliminaries}
\subsection{Holomorphic Riemannian manifolds}
A \textbf{holomorphic Riemannian manifold} $(\mathbb{X},\inners)$ is the data of a complex manifold $\mathbb{X}$ equipped with a non-degenerate and holomorphic symmetric $\C$-bilinear form $\inners$ on the holomorphic tangent bundle of $\mathbb{X}$, which we canonically identify with its (smooth) tangent bundle. The pairing $\inners$ is referred to as the \textbf{holomorphic Riemannian metric}. Holomorphic Riemannian manifolds (in a way that can be made precise) ``complexify" the theory of Riemannian manifolds; on a holomorphic Riemannian manifold, we have analogues of objects from Riemannian geometry: \textbf{Levi-Civita} (holomorphic) affine connections (a torsion free holomorphic connection that makes $\inners$ parallel), curvature tensors, geodesic (complex)
curves, sectional curvatures on non-degenerate $2$-planes, etc. See \cite{BEE} for an account of the theory.

The basic example of a holomorphic Riemannian manifold is $\mathbb C^n$ with a flat holomorphic Riemannian metric $\inners$ induced by a complex bilinear form on $\C^n$. It is proved in \cite[Theorem 2.6]{BEE} that simply connected holomorphic space forms, such as $(\mathbb C^n,\inners)$, are uniquely determined up to isometry by their curvature. Along with this example, we have the $n$-dimensional hyperbolic space form of constant curvature $-1$,  concretely modeled as the hyperboloid
$$ \hat {\mathbb X}_{n}=\{\underline z \in \C^{n+1}\ |\ \langle \underline z, \underline z\rangle_{n,1}:= - z_{n+1}^2+ \sum_{j=1}^n z_j^2 =-1\},$$ which inherits a holomorphic metric from the inclusion into $(\C^{n+1}, \inners_{n,1}).$ It is standard to identify $\hat {\mathbb X}_n$ with $\mathrm{SO}(n+1,\C)/\mathrm{SO}(n,\C)$, equipped with a suitably normalized Killing form.

 In this paper, we're mostly interested in the $2$-dimensional hyperbolic space form $\hat {\mathbb X}_{2},$ which admits another geometric characterization  as the space of oriented geodesics of hyperbolic space $\mathbb{H}^3$. To see this: $\mathrm{SO}(3,\C)$ acts transitively on the space of oriented geodesics, and the stabilizer of any oriented geodesic is conjugate to $\mathrm{SO}(2,\C)$. Identifying geodesics in $\mathbb{H}^3$ with their endpoints in $\mathbb{CP}^1$ determines a biholomorphism $$\hat{\mathbb X}_2\to \mathbb{G}:=\mathbb{CP}^1\times \mathbb{CP}^1\backslash \Delta.$$ Under this biholomorphism, the pushed-forward holomorphic Riemannian metric has the following description: given any complex affine chart $(V,z)$ on $\CP^1$, the metric on $(V\times V\setminus \Delta, (z_1,z_2))$  has the form \[
 \inners= -\frac 4 {(z_1-z_2)^2} dz_1\cdot dz_2 \ 
 \]
(see \cite[Section 2.4]{BEE}). In this model, the isometry group of $(\mathbb G, \inners)$ is generated by the (diagonal) action of $\PSL(2,\C)$ and the swap $(z_1, z_2)\mapsto (z_2,z_1)$.

The space $\mathbb{G}$ comes with two types of complex geodesics: the non-degenerate ones, isometric to $(\C^*, \frac{dz^2}{z^2})$, and the degenerate ones, which are parametrized by $\C$ and on which the metric vanishes. Identifying with $\mathrm{SL}(2,\C)/\mathrm{SO}(2,\C)$, one obtains the non-degenerate geodesics by exponentiating the $\C$-span of a semisimple matrix in the Lie algebra, and a degenerate one by exponentiating the $\C$-span of a nilpotent matrix.

\subsection{Complex metrics on surfaces}
 Throughout the paper, given a smooth manifold $M,$ we denote the complexified tangent bundle $TM\otimes_{\R}\C$ by $\C TM.$ Similarly, the complexified cotangent bundle is denoted $\C T^*M$. 
\begin{defn}
    A \textbf{complex metric} on a smooth (real) manifold $M$ is a non-degenerate and smooth symmetric $\C$-bilinear form on $\C TM$.
\end{defn}
The first example of a complex metric is the bilinear extension of a Riemannian metric to the complexified tangent bundle. Throughout the paper, we will view a complex metric on a manifold $M$ as the same thing as a $\pi_1(M)$-invariant complex metric on the universal cover $\widetilde{M},$ and we won't distinguish notation.

Complex metrics were studied in detail in \cite{BEE}. Similar to holomorphic Riemannian manifolds, basic constructions in Riemannian geometry extend to the setting of complex metrics. Most of the proofs apply verbatim, and the formulas are the same, except we're working on complex bundles (see \cite{BEE} for details). For instance, every complex metric $g$ has a \textbf{Levi-Civita connection} (a torsion-free affine connection on $\C TM$ that makes $g$ parallel), gradients, Laplacians (denoted $\Delta_g$), curvature tensors, sectional curvature on non-degenerate $2$-planes, etc. 

Now let $\Sigma$ be a surface. A complex metric $g$ on $\Sigma$ determines a smooth splitting of $\C T\Sigma$ into $g$-isotropic lines $\C T\Sigma=L_1\oplus L_2$. To make the definition below, we point out that the projectivized real bundle $\mathbb P T_p \Sigma$ embeds $\mathbb P \C T_p\Sigma$, with each fiber embedding as an equatorial circle that cuts the sphere into two connected components.

\begin{defn}\label{def: positive}
    A complex metric $g$ is \textbf{positive} if, at each point of $\Sigma$, the associated isotropic directions $L_1,L_2\in \mathbb P (\C T\Sigma)$ do not lie in $\mathbb P T \Sigma$ and have images in distinct components of $\mathbb P 
(\C T\Sigma)\backslash \mathbb P(T \Sigma).$

\end{defn}

An orientation on $\Sigma$ fixed, the induced orientation of the circle $\mathbb P T_p\Sigma\subset \mathbb {P}(\C T_p\Sigma)$ allows us to label the two connected components of $ \mathbb {P}(\C T_p\Sigma)\setminus \mathbb P T_p\Sigma$ as the ``inner" and the ``outer" disk. In the following, we will assume that $L_1$ and $L_2$ lie in the inner and in the outer disk respectively.

The most important observation to make about a positive complex metric $g$ is that its conformal class defines a pair $\ccpair$ of oppositely oriented complex structures. More precisely, the isotropic directions $L_1$ and $L_2$ determine oppositely oriented almost complex structures $J_1$ and $J_2$ on $\Sigma$ as follows: the line sub-bundle $L_1$ (resp. $L_2$) is declared to be the $(-i)$-eigenline for $J_2$ (resp. $J_1$). One can easily see that $J_1$ and $-J_2$ are consistent with the orientation. 
If $z$ and $\overline w$ are local coordinates for the almost complex structures $J_1$ and $J_2$ respectively, then $L_1=Span_{\mathbb C} (\partial_w)$ and $L_2=Span_{\mathbb C}(\partial_{\overline z})$, and $g$ locally has the form $g=\lambda dzd\overline w.$

Note that Riemannian metrics complexify to positive complex metrics because the isotropic lines will be antipodal. Moreover, in this case we have $J_2=-J_1$, which is equivalent to saying that the associated complex structures are conjugate.

\subsubsection{Admissible immersions and Bers metrics}
Let $\widetilde{\Sigma}$ be the universal cover of $\Sigma$. We say that an immersion $f\colon\widetilde \Sigma\to \mathbb G$ is \textbf{admissible} if it is totally real. Equivalently, for each point $p\in \Sigma$, the complexification of the derivative realizes an isomorphism from $\C T_p\widetilde \Sigma$ to $T_{f(p)}\mathbb G$, which in turn is the same as saying that $g=f^*\inners$ is a complex metric. It turns out that all complex metrics of constant curvature $-1$ arise in this fashion.
\begin{prop}[Corollary 4.2 from \cite{BEE}]\label{prop: metric to immersion}
Let $g$ be a complex metric on $\Sigma$. There exists an equivariant admissible immersion $f\colon \widetilde \Sigma\to \mathbb G$ such that $g=f^*\inners$ if and only if $g$ has constant curvature $-1$. The map $f$ is unique up to post composing with isometries $(\mathbb G, \inners)$.
\end{prop}

In the proposition above, $g$ is positive if and only if, upon identifying $\mathbb{G}$ with $\mathbb{CP}^1\times \mathbb{CP}^1\backslash \Delta$, the two projections to $\mathbb{CP}^1$ are local diffeomorphisms with opposite orientations (see \cite[Proposition 4.9]{BEE}). Moreover, when $g$ can be deformed to a positive metric through constant curvature $-1$ metrics, the map $f$ is equivariant for a representation $\rho\colon \pi_1(\Sigma)\to \PSL(2,\C)$, called the \textbf{holonomy} of $f$. We will often also refer to the class of $\rho$ in the character variety as the holonomy of $g$.

The most important complex metrics arising from the pull-back construction are Bers metrics, which arise from Bers' Simultaneous Uniformization Theorem (see \cite{SU}). One formulation of Bers' Theorem is that, for any pair of oppositely oriented complex structures $c_1$ and $\overline{c_2}$ on $\Sigma$, we have a quasi-Fuchsian representation $\rho$ and a pair of $\rho$-equivariant biholomorphisms $f_1:(\widetilde{\Sigma},c_1)\to \Omega_\rho^+$ and $\overline{f_2}:(\widetilde{\Sigma},\overline{c_2})\to \Omega_\rho^-$, where $\Omega_\rho^+$ and $\Omega_\rho^-$ are disjoint domains in $\mathbb{CP}^1$ whose boundaries coincide and are a $\rho$-invariant quasicircle. We put $f_1$ and $\overline{f_2}$ together to obtain a $\rho$-equivariant admissible immersion $(f_1,\overline{f_2}):\widetilde{\Sigma}\to \mathbb{G}$.
The bilinear form
\[
h=(f_1, \overline{f_2})^*\inners
\]
is a $\pi_1(\Sigma)$-invariant positive complex metric on $\widetilde \Sigma$ -- hence on $\Sigma$ -- that we will call a \textbf{Bers metric}. Necessarily, $h$ induces the complex structures $(c_1, \overline{c_2})$, and from the proposition above, $h$ has constant curvature $-1$. When $f_1=f_2$ (necessarily $c_1=c_2=:c$), $h$ is the complexification of a Riemannian metric of constant curvature $-1$ in the conformal class determined by $c$.

 If $\Sigma$ is closed, the data $(\rho,f_1,\overline{f_2})$ is unique up to composing $f_1$ and $\overline{f_2}$ and conjugating $\rho$ with the same element in $\PSL(2,\C)$; for $\Sigma$ not closed, they are unique up to a suitable joint normalization (see \cite{ESholodependence}). 
 The map $(c_1, \overline{c_2})\mapsto \rho$ takes image in the open subset of homomorphisms into $\PSL(2,\C)$ corresponding to quasi-Fuchsian representations. Moreover, it descends to a biholomorphism
\[
\TSTS\xrightarrow{\sim} \QF(S)\subset \chi^{\mathrm{an}}(\pi_1(S), \PSL(2,\C)).
\]
By the above discussion, when $\Sigma$ is closed, $(c_1,\overline{c_2})$ are induced by a unique Bers metric, which we denote by $\hpair$. It is further proved in \cite{BEE} that Bers metrics form a connected component of the subset of smooth complex metrics of constant curvature $-1$, and $\ccpair \mapsto \vb*h{(c_1,\overline{c_2})}$ defines a bijective (in fact, holomorphic by \cite{ESholodependence}) correspondence between pairs of complex structures and Bers metrics.

\subsubsection{Isotropic frames} \label{subsubsec: connection forms}
In the proof of Theorem \ref{thm: Bochner equation and harmonicity} below, we'll make use of isotropic frames for complex metrics together with the related curvature formula.
\begin{prop}\label{prop: connection forms}
    Let $g$ be a complex metric on $\Sigma$, and let $\eta^1,\eta^2$ be local complex-valued 1-forms such that $g=\eta^1\cdot \eta^2$. Then, there exists a unique locally defined 1-form $\eta_c$ such that 
\begin{equation}
\label{eq: isotropic LC forms}
    \begin{cases}
  d\eta^1=-\eta_c\wedge \eta^1  \\
  d\eta^2=\eta_c\wedge \eta^2,
\end{cases}
\end{equation}
and, if $K_g$ is the sectional curvature, $$d\eta_c= \frac{K_g}{2}\eta^1\wedge \eta^2.$$ 
\end{prop}
This mirrors a standard construction for Riemannian metrics; we include the proof for completeness.
\begin{proof}
There are unique locally defined functions $A$ and $B$ such that $d\eta^1 = A\eta^1\wedge \eta^2$ and $d\eta^2 = B \eta^1\wedge \eta^2$. Define $$\eta_c=B\eta^1+A\eta^2.$$ Then, $d\eta^1=A\eta^1\wedge \eta^2= -A\eta^2\wedge \eta^1=-\eta_c\wedge \eta^1$, and similarly $d\eta^2= \eta_c\wedge \eta^2$, i.e., (\ref{eq: isotropic LC forms}) holds. Moreover, $\eta_c$ is the unique $1$-form satisfying \eqref{eq: isotropic LC forms}: for any other such $1$-form $\hat{\eta}_c$, $$(\eta_c-\hat{\eta}_c)\wedge \eta^1=(\eta_c-\hat{\eta}_c)\wedge \eta^2=0,$$ which implies $\eta_c-\hat{\eta}_c$ is parallel to both $\eta^1$ and $\eta^2$, and then linear independence of $\{\eta^1,\eta^2\}$ forces $\eta_c=\hat{\eta}_c$.

For the curvature formula, let $(E_1,E_2)$ be the dual frame of the complexified tangent bundle. This frame is isotropic for $g$, and $g(E_1,E_2)=\frac 12 $. The Levi-Civita connection $\nabla^g$ satisfies $\nabla^g E_1 = \eta_c \otimes E_1$ and $\nabla^g E_2 = -\eta_c\otimes E_2$. Indeed, it is routine to check that these expressions define a compatible and torsion free connection, which by uniqueness must agree with $\nabla^g$. Now, we calculate
\begin{align*}
R^g(X,Y)E_1
&=
\nabla^g_X\bigl(\eta_c(Y)E_1\bigr)
-\nabla^g_Y\bigl(\eta_c(X)E_1\bigr)
-\eta_c([X,Y])E_1 \\
&=
\bigl(
X(\eta_c(Y))-Y(\eta_c(X))-\eta_c([X,Y])
\bigr)E_1 
+\eta_c(Y)\eta_c(X)E_1
-\eta_c(X)\eta_c(Y)E_1 \\
&=
d\eta_c(X,Y)E_1.
\end{align*}
An analogous calculation returns $R^g(X,Y)E_2=-d\eta_c(X,Y)E_2.$ Hence, using $g(E_1,E_2)=\frac{1}{2}$ and the isotropy condition, $$K_g=\frac{g(R^g(E_1,E_2)E_2,E_1)}{-g(E_1,E_2)^2}=\frac{g(-d\eta_c(E_1,E_2)E_2,E_1)}{-g(E_1,E_2)^2}=\frac{-d\eta_c(E_1,E_2)g(E_1,E_2)}{-g(E_1,E_2)^2}=2d\eta_c(E_1,E_2).$$ We conclude that $d\eta_c=\frac{K_g}{2}\eta^1\wedge \eta^2.$
\end{proof}

\section{Complex harmonic maps to $\mathbb G$}\label{sec: c harmonic maps}
Throughout, let $\Sigma$ be an oriented surface, not necessarily closed.

\subsection{Definition}\label{sec: definition}
Let $(c_1,\overline{c_2})$ be oppositely oriented complex structures on $\Sigma$.  We decompose the complexified cotangent bundle $\C T^*\Sigma$ as $\C T^*\Sigma = (T^*\Sigma)_{c_1}^{1,0}\oplus (T^*\Sigma)_{c_2}^{0,1}$ and write $\Pi_+$ for the projection onto the first factor and $\Pi_-$ for the projection onto the second. Let $(\mathbb{X},\inners)$ be a holomorphic Riemannian manifold with Levi-Civita connection $D$. The pullback connection
$\nabla^f\colon \Gamma(f^*T\mathbb X)\to \Gamma(T^* \Sigma\otimes f^*T\mathbb X)$ extends uniquely to a map
$$\nabla^f\colon \Gamma(f^*T\mathbb X)\to \Gamma(\C T^* \Sigma\otimes f^*T\mathbb X)$$ 
satisfying the $\C$-linearity condition $\nabla^f_{X+iY}=\nabla^f_X+ i\nabla^f_Y$ for all $X,Y\in \Gamma(\C T\Sigma)$ and compatible with the pull-back metric $f^*\inners$ on $f^*T\mathbb X$. $\nabla^f$ determines an exterior covariant differential 
$$d^f\colon \Omega^1(\Sigma,f^*T\mathbb X)\to \Omega^2(\Sigma,f^*T\mathbb X)$$
defined by the standard property
$$d^f(\alpha\otimes \xi)= d\alpha\otimes \xi -\alpha \wedge \nabla^f\xi$$
for all $\alpha\in \Gamma(\C T^*\Sigma)$ and $\xi\in \Gamma(f^*T\mathbb X)$. 
\begin{defn}\label{def: charmonic}
    We say that a $C^2$ map $f:(\Sigma,c_1,\overline{c_2})\to (\mathbb{X},\inners)$ is \textbf{complex harmonic} if $d^f(\Pi_+\circ df)=0$.
\end{defn} 
\begin{remark}
\label{rmk: equivalent harmonicity conditions}
It is proved in \cite[Section 4.1]{ES_Chmaps} that $d^f(\Pi_+\circ df) = -d^f(\Pi_-\circ df) $, and that the vanishing of either of these sections is equivalent to $\textrm{tr}_h((\nabla^{h^*}\otimes \nabla^f) (df))=0$, where $h$ is any positive complex metric inducing complex structures $(c_1,\overline{c_2})$, and $\nabla^{h^*}$ is the dual of the Levi-Civita connection on $\C T^*\Sigma$. We can therefore define complex harmonicity as the vanishing of any of these three objects.
\end{remark}

The first example to have in mind is an ordinary harmonic map: when $c:=c_1=c_2,$ the above splitting of the complexified tangent bundle is just the ordinary splitting into types as $\C T^*\Sigma = (T^*\Sigma)_{c}^{1,0}\oplus (T^*\Sigma)_{c}^{0,1}$, and the equation above becomes the familiar harmonic map equation $d^f(df^{1,0})=0,$ with the only difference being that we're now mapping into a holomorphic Riemannian manifold. If $\mathrm X\subset \mathbb{X}$ is a totally real and totally geodesic Riemannian submanifold, then any harmonic map from $(\Sigma,c)$ into $\mathrm X$ is also a complex harmonic map from $(\Sigma,c,\overline{c})$ into $\mathbb{X}$.

\begin{remark}
\label{rmk: harmonicity in coords}
    In local holomorphic coordinates $z$ and $\overline{w}$ for $c_1$ and $\overline{c_2}$, for any $1$-form $\omega$, $$\Pi_+\omega = \frac{\omega(\partial_{w})}{\partial_{w}z}dz, \hspace{1mm} \Pi_-\omega = \frac{\omega(\partial_{\overline{z}})}{\partial_{\overline{z}}\overline{w}}d\overline{w}.$$
    By evaluating $d^f(\Pi_+ \circ df)$ with respect to the two induced local frames for $\C T\Sigma$, $(\partial_{z}, \partial_{\overline {z}})$ and $(\partial_{w}, \partial_{\overline {w}})$, we get that 
    \begin{equation}
    \label{eq: harmonicity in coords 1}
    (d^f(\Pi_+ df)) (\partial_{z}, \partial_{\overline {z}})= - \nabla^f_{\partial_{\overline {z}}} ((\Pi_+ df)(\partial_{z})= -\nabla^f_{\partial_{\overline {z}}}\frac{df(\partial_{w})}{\partial_{w} z},
    \end{equation}
    and similarly, 
    \begin{equation}
    \label{eq: harmonicity in coords 2}
     (d^f(\Pi_- df)) (\partial_{w}, \partial_{\overline {w}})= - \nabla^f_{\partial_{ w}} ((\Pi_+ df)(\partial_{\overline {w}})= -\nabla^f_{\partial_{ w}}\frac{df(\partial_{\overline {z}})}{\partial_{\overline {z}} \overline {w}},
    \end{equation}
    where we used that $(\Pi_+ df)(\partial_{\overline {z}})= 0=(\Pi_- df)(\partial_{ {w}})$.
    Thus, complex harmonicity can be expressed in coordinates as $\nabla^f_{\partial_{\overline {z}}}(\frac 1 {\partial_{w} z}{df(\partial_{w})})=0$ or $\nabla^f_{\partial_{ w}}(\frac1 {\partial_{\overline {z}} \overline {w}}{df(\partial_{\overline {z}})})=0$.
\end{remark}
We mention in passing that, by analogy with the real theory, we have a notion of complex Dirichlet energy of a $C^2$ map $f$, $\int_{\Sigma} \langle \Pi_+ df,\Pi_- df\rangle$, and complex harmonic maps are critical points. We won't have much use for a variational standpoint in this paper, but this is the Dirichlet energy considered in Corollary \ref{cor: complex energy}.

Analogous to the case of ordinary harmonic maps, we have a notion of Hopf differential. For any $C^2$ map $f:(S,c_1,\overline{c_2})\to (\mathbb{X},\inners)$ we can decompose
\begin{equation}\label{pullbackmetric}
    f^*\inners = 2\langle \Pi_+d f, \Pi_-d f\rangle + \langle \Pi_+d f, \Pi_+d f\rangle+\langle \Pi_-d f, \Pi_-df\rangle. 
\end{equation}
As in \cite{ES_Chmaps}, the tensors $q_1:=\langle \Pi_+ d f, \Pi_+ d f\rangle$ and $\overline{q_2}:=\langle  \Pi_-d f,  \Pi_-df\rangle$ will be called the \textbf{Hopf differentials} of $f$. 
\begin{prop}\label{prop: hopf}
    If $f$ is complex harmonic, $q_1$ and $\overline {q_2}$ are holomorphic quadratic differentials for $c_1$ and $\overline{c_2}$ respectively. When $\mathbb{X}$ has complex dimension $2$ and $f$ is admissible in an open and dense subset of $\Sigma$, the converse holds. 
\end{prop}
\begin{proof}
    Let $z$ and $\overline w$ be local holomorphic coordinates for $c_1$ and $\overline{c_2}$ respectively.
    By metric compatibility and Equation \eqref{eq: harmonicity in coords 1},
    \begin{align*}
    \partial_{\overline z} (q_1(\partial_z, \partial_z)) &= \partial_{\overline z} (f^*\inner{(\Pi_+ df)(\partial_z), (\Pi_+ df)(\partial_z)}) = 2\inner{\nabla^f_{\partial_{\overline z}} ((\Pi_+ df)(\partial_z)), (\Pi_+ df)(\partial_z) }\\
    =& -2 \inner{(d^f(\Pi_+df))(\partial_z, \partial_{\overline z}), (\Pi_+df)(\partial_z)  }.
     \end{align*}
     Similarly, since $d^f(\Pi_+ df)=-d^f(\Pi_- df)$ (Remark \ref{rmk: equivalent harmonicity conditions}),
     \[
     \partial_w (\overline{q_2}(\partial_{\overline w}, \partial_{\overline w}))= -2 \inner{(d^f(\Pi_-df))(\partial_{\overline w}, \partial_{w}), (\Pi_-df)(\partial_{\overline w})  }= 2 \inner{(d^f(\Pi_+df))(\partial_{\overline w}, \partial_{w}), (\Pi_-df)(\partial_{\overline w})  }.
     \]
     Hence, if $d^f(\Pi_+ df)=0$, $q_1$ and $\overline{q_2}$ are holomorphic for $c_1$ and $\overline{c_2}$ respectively.
     
      For the converse, observe that if $f$ is admissible and $\mathbb{X}$ has complex dimension $2$, then for all $p\in \Sigma$, $df(\C T_p \Sigma)=T_{f(p)}\mathbb{X}$, hence, with $z$ and $\overline{w}$ as above, $T_p\mathbb{X}$ is spanned by the vectors $((\Pi_+ df)(\partial_z)= \frac{1 }{\partial_w z} df(\partial_w)$ and $((\Pi_- df)(\partial_{\overline w})= \frac{1 }{\partial_{\overline z} {\overline w}} df(\partial_{\overline z})$. 
      Since $(d^f(\Pi_+df))(\partial_{\overline w}, \partial_{w})$ and  $(d^f(\Pi_+df))(\partial_z,\partial_{\overline z})$ are multiples, the computation above shows that, if $q_1$ and $\overline {q_2}$ are holomorphic differentials, then $f$ is complex harmonic. Finally, if $f$ is admissible in just an open and dense set, the result continues to hold due to continuity of second derivatives.
\end{proof}

In the sequel, we will be mostly concerned with equivariant maps from the universal cover $\widetilde{\Sigma}$. It is easily verified that in this situation, the Hopf differentials and $\langle \Pi_+d f, \Pi_-d f\rangle$ descend to $\Sigma$. In this context, complex energy is defined by integrating the descended function on $\Sigma$.

\subsection{Bochner equation}\label{subsec: Bochner}

We continue with $\Sigma$ and $(c_1,\overline{c_2})$ as in the previous subsection. We prove in Theorem \ref{thm: Bochner equation and harmonicity} below that solutions to a complexified version of the classical Bochner equation produce complex harmonic maps with prescribed Hopf differentials. Theorem \ref{thm: Bochner equation and harmonicity} is a special case of Theorem C from \cite{ES_Chmaps}, which concerns more general harmonic maps to holomorphic symmetric spaces, and goes through the Higgs bundle machinery. Here we give a classically flavoured and more down-to-earth proof of Theorem \ref{thm: Bochner equation and harmonicity}, using Proposition \ref{prop: metric to immersion}. A proof that real harmonic maps satisfy the ordinary Bochner equation can be found in \cite{SchoenYauUnivalent}; we could not locate a proof of the converse direction (not using Higgs bundles) in the literature.

To motivate the statement of Theorem \ref{thm: Bochner equation and harmonicity}, we recall that for any ordinary harmonic map $f:(\widetilde{\Sigma},c)\to \mathbb{H}^2$ and Riemannian metric $h$ compatible with $c$, we can write the pullback metric as $$e(f)h + q +\overline{q},$$ where $e(f)$ is a function.
The function $e(f)$ splits as $$e(f)=||\partial f||^2+||\overline{\partial} f||^2,$$ and the map $f$ is an orientation-preserving immersion if and only if $||\partial f||>||\overline{\partial} f||$ everywhere (see \cite{SchoenYauUnivalent}). In this case, setting $U:=||\partial f||$, the pullback metric is rewritten as 
\begin{equation}\label{eq: real pullback metric}
    \left(U +\frac 1 U \frac{|q|^2}{ h^2}\right) h +q+\overline{q}.
\end{equation}

In Theorem \ref{thm: Bochner equation and harmonicity}, we look for complex harmonic maps with pullback metrics analogous to \eqref{eq: real pullback metric}.
For the statement, note the following: if $q_1=\phee dz^2$, $\overline{q_2}=\overline \psi d\overline w^2$ are quadratic differentials and $h=\lambda dzd\overline w$ is a complex metric, then $\frac{q_1\overline{q_2}}{h^2}$ defines a function.

\begin{thm}
\label{thm: Bochner equation and harmonicity}
Let $c_1$ and $\overline{c_2}$ be oppositely oriented complex structures on $\Sigma$ with Bers metric $h$. Let $q_1$ be a $c_1$-holomorphic quadratic differential and let $\overline{q_2}$ be a $\overline{c_2}$-holomorphic quadratic differential. Suppose there exists a function $U: \Sigma \to \C^*$ such that     
\begin{equation}
\label{eq: Bochner}
\Delta_h \log U= 2U-2\frac{q_1\overline{q_2} }{h^2}\frac 1 U-2.
\end{equation} Then, there exists an equivariant complex harmonic and admissible immersion $f\colon (\widetilde \Sigma,c_1,\overline{c_2})\to \mathbb G$, unique up to isometries of $\mathbb{G}$, with pull-back metric 
\begin{equation}\label{eq: pullback, complex}
    g=q_1+ \left(U +\frac 1 U \frac{q_1\overline{q_2}}{ h^2}\right) h +\overline{q_2}.
    \end{equation}
Conversely, if we have an equivariant complex harmonic map with pullback metric (\ref{eq: pullback, complex}), then the function $U$ solves (\ref{eq: Bochner}).
\end{thm}
Before giving the proof, we need a small lemma.
\begin{lem}\label{lem: det}
    For a complex bilinear form expressed as $g=q_1+ \left(U +\frac 1 U \frac{q_1\overline{q_2}}{ h^2}\right) h +\overline{q_2}$, $g$ is non-degenerate if and only if $\left(U -\frac 1 U \frac{q_1\overline{q_2}}{ h^2}\right)$ never vanishes.
\end{lem}
\begin{proof}
    Local holomorphic coordinates $z$ and $\overline{w}$ for $c_1$ and $\overline{c_2}$ respectively determine the coframe $(dz,d\overline{w})$ in which the matrix of $g$ is expressed $$\begin{pmatrix}
        q_1 & \frac{1}{2}\left(U +\frac 1 U \frac{q_1\overline{q_2}}{ h^2}\right)h \\
        \frac{1}{2}\left(U +\frac 1 U \frac{q_1\overline{q_2}}{ h^2}\right)h & \overline{q_2}
    \end{pmatrix}.$$
    The determinant is $\left(U +\frac 1 U \frac{q_1\overline{q_2}}{ h^2}\right)^2h^2-4q_1\overline{q_2}=-\frac{h^2}{4}\left(U -\frac 1 U \frac{q_1\overline{q_2}}{ h^2}\right)^2.$
\end{proof}
We will also repeatedly use the formula, for any $C^1$ function $v$,
\begin{equation}
\label{eq: differential in dz dw bar}
    dv= \frac {\partial_{w}v} {\partial_{w}z}  dz + \frac {\partial_{\overline {z}}v} {\partial_{\overline {z}}\overline {w}}  d\overline{w}.
\end{equation}

\begin{proof}[Proof of Theorem \ref{thm: Bochner equation and harmonicity}]
    By Propositions \ref{prop: metric to immersion}, Theorem \ref{thm: Bochner equation and harmonicity} is equivalent to the statement that $g$ has constant curvature $-1$ if and only if Equation \eqref{eq: Bochner} holds.

In local coordinates $z$ and $\overline{w}$ for $c_1$ and $\overline{c_2}$, write $h=\lambda dzd\overline{w}$, $q_1=\varphi dz^2$, $\overline{q_2}=\overline{\psi}d\overline{w}^2$. We recall from \cite[Sections 3.3 and 3.4]{ES_Chmaps} the local coordinate expression $$\Delta_h = \frac 1 {\partial_{\overline {z}} \overline{w}}\partial_{\overline {z}}(\frac{1}{\partial_{w}z} \partial_{w})$$ and that $\Delta_h \log \lambda=2$.
In our coordinates, define the functions 
\begin{align*}
P&= \lambda U- \frac{\phee \overline \psi}{\lambda}\frac 1 U,\\
Q&=\frac{1}{\partial_{w}z} \partial_{w} \log (\lambda U).
\end{align*}
Since $\Delta_h \log \lambda=2$, Equation \eqref{eq: Bochner} can be rewritten as 
\[
\Delta_h \log(\lambda U)= \Delta_h \log\lambda + \Delta_h\log U=2+2\left(\frac 1 \lambda P-1\right)= 2\frac P \lambda.
\]
By our expression for $\Delta_h$, Equation \eqref{eq: Bochner} is thus equivalent to
\[
\frac 1 {\partial_{\overline {z}} \overline{w}}\partial_{\overline {z}} Q= \frac 1 2 P.
\]
We will show that this last equation is equivalent to $K_g=-1$. To this end, we find an expression involving the curvature of $g$ and our other objects. Preparing to apply Proposition \ref{prop: connection forms}, we write $g=\eta^1\cdot \eta^2$, where 
\begin{align*}
  \eta^1&= dz+\overline{\psi} \frac 1 {\lambda U} d\overline {w},\\
    \eta^2&= \phee dz+\lambda Ud\overline {w}.
\end{align*}
By a straightforward calculation, using 
Equation \eqref{eq: differential in dz dw bar} where needed, we get:
\begin{align*}
    \eta^1\wedge \eta^2&= P dz\wedge d\overline {w}\\
    d\eta^1&=d(\frac 1 {\lambda U})\wedge  \overline{\psi}d\overline {w} = - \frac {\overline{\psi}} {\lambda U} d (\log(\lambda U))\wedge  d\overline{w}= - \frac {\overline{\psi}} {\lambda U}Q dz\wedge d\overline{w}=- \overline \psi \frac{Q}{\lambda UP} \eta^1\wedge \eta^2\\
     d\eta^2&=d(\lambda U)\wedge d\overline {w}= \lambda U d(\log(\lambda U))\wedge d\overline {w}=  \lambda UQ dz\wedge d\overline{w}=\frac{\lambda UQ}P \eta^1\wedge \eta^2,
\end{align*}
where the division by $P$ is permitted thanks to Lemma \ref{lem: det}. Now let $\eta_c$ be the connection $1$-form associated with $\{\eta^1,\eta^2\}$ via Proposition \ref{prop: connection forms}. According to the definition of $\eta_c$,
\[
\eta_c= \frac{\lambda UQ}P \eta^1 - \overline \psi \frac{Q}{\lambda UP}\eta^2=
\frac{\lambda UQ} P dz +\overline \psi \frac Q P d\overline {w} - \frac{\phee\overline \psi Q}{\lambda U P}dz-  \overline \psi \frac Q P d\overline {w}=Q dz.
\]
{As well, by Equation \eqref{eq: differential in dz dw bar} again, 
$$d\eta_c = -\frac 1 {\partial_{\overline {z}} \overline{w}}\partial_{\overline {z}} Q dz\wedge d\overline w.$$
}
Hence, by Proposition \ref{prop: connection forms},
\[ \frac 12 K_g \cdot P dz\wedge d\overline {w}= \frac 12 K_g \eta^1\wedge \eta^2= d\eta_c= -{\frac 1 {\partial_{\overline {z}} \overline{w}}}\partial_{\overline {z}}Q dz\wedge d\overline {w} ,\]
which shows that Equation \eqref{eq: Bochner} is equivalent to $K_g=-1$.
\end{proof}
\begin{remark}
    Both Theorem \ref{thm: Bochner equation and harmonicity} and \cite[Theorem C]{ES_Chmaps} (see the beginning of the section) produce complex harmonic maps from (\ref{eq: Bochner}), but it's not immediately apparent that the complex harmonic maps are isometric. To see that they are, one just has to observe that the pullback metrics are the same. To obtain the pullback metric coming from \cite[Theorem C]{ES_Chmaps}, see \cite[Theorem 4.11]{ES_Chmaps} and substitute $\Pi_+ \circ df$ (from the current paper) for $\phi_1$ (from \cite{ES_Chmaps}), $\Pi_-\circ df$ (from the current paper) for $\overline{\phi_2}^{\mathcal{I}}$ (from \cite{ES_Chmaps}), and $\inners$ (from the current paper) for $\frac{1}{2}\textrm{tr}(\cdot, \cdot)$ (from \cite{ES_Chmaps}).
\end{remark}

\subsection{Examples}\label{subsec: examples}
Any (ordinary) harmonic map to a hyperbolic surface lifts to an equivariant harmonic map to the hyperbolic plane $\mathbb{H}^2$. The first examples of complex harmonic maps to $\mathbb{G}$ are ordinary harmonic maps to $\mathbb{H}^2$ post-composed with the inclusion $\mathbb{H}^2\to\mathbb{G}$. One can then take deformations thereof, as we'll do in the next section. More examples are below. 

\subsubsection{Minimal immersions in $\mathbb{H}^3$}\label{subsubsec: mimimal immersions} Input the data $c:=c_1=c_2$, $q:=q_1=-q_2$ into Theorem \ref{thm: Bochner equation and harmonicity}. The equation (\ref{eq: Bochner}) agrees with the Gauss equation for finding a minimal immersion $\widetilde{\Sigma}\to \mathbb{H}^3$ with first fundamental form $\Is$ in the conformal class specified by $c$, and shape operator ${B}$ inducing second fundamental form $\II:=\Is(\mathtt B\cdot,\cdot) = \frac 1 2 (q +\overline{q})$ (see \cite{Uh}). Indeed, let $h$ be the Riemannian hyperbolic metric in the conformal class of $c$, and suppose that, for a real-valued function $U$,  $\mathrm I=U \cdot h$ and $\II= \frac 1 2 (q +\overline{q})$ provide a minimal immersion $\sigma$. Then $U$ solves both the Gauss equation and (\ref{eq: Bochner}). By Theorem \ref{thm: Bochner equation and harmonicity}, $U$ also engenders a complex harmonic map to $\mathbb{G}$. We explain that this complex harmonic map is related to $\sigma$ via a Gauss map.

For an oriented immersion $\sigma: \widetilde \Sigma\to \Hyp^3$ with normal vector field $N(p)$, we can define a normal map from $\widetilde{\Sigma}$ to the unit tangent bundle $T^1{\Hyp^3}$ by $p\mapsto (\sigma(p),N(p))$. Under the interpretation of $\mathbb{G}$ as the space of oriented geodesics, we have the Gauss map
$\widehat\sigma_0: \widetilde \Sigma\to \mathbb G$ that assigns to each $p$ the oriented line determined by $(\sigma(p),N(p))$. We can show without using Theorem \ref{thm: Bochner equation and harmonicity} that $\widehat{\sigma}$ is complex harmonic. From \cite[Proposition 4.12]{BEE}, for our minimal map $\sigma,$ the pull-back through the Gauss map $\widehat \sigma$ of the holomorphic Riemannian metric of $\mathbb G$ is \[
\widehat{\sigma}^*\inners= \Is((id+iJ\mathtt B)\cdot,(id+iJ\mathtt B)\cdot)= \Is - \Is(\mathtt B^2\cdot,\cdot)+ q-\overline q,
\]
where $J$ is the almost complex structure associated with $c$. By \cite[Proposition 4.12]{BEE}, $\widehat \sigma^*\inners$ is a complex metric if and only if the curvature of $\Is$ is nowhere zero. By the Gauss equation, the curvature of $\Is$ is less than $-1$ \cite[Theorem 4.2]{Uh}, and hence $\widehat \sigma^*\inners$ is indeed a complex metric. Consequently, $\widehat \sigma$ is admissible. Since $\mathtt B$ is traceless, $\Is-\Is(\mathtt B^2\cdot, \cdot)$ is a conformal multiple of $\Is$, and thus the holomorphicity of $q$ ensures that $\widehat{\sigma}$ is harmonic.

To relate $\widehat{\sigma}$ with the complex harmonic map from Theorem \ref{thm: Bochner equation and harmonicity}, say, $f$, note that if $h$ is the conformal hyperbolic metric on $(\Sigma,c)$, then $$\Is-\Is(\mathtt{B}^2\cdot,\cdot) = (U + \frac 1 U|q|_h^2)h$$ (see \cite{Uh}). Thus, the two maps have the same intrinsic data, and Proposition \ref{prop: metric to immersion} implies that $\hat{\sigma}$ and $f$ are isometrically identified. 

\subsubsection{Developments of $\mathbb{CP}^1$ structures}
\label{subsubsection: CP1 structures}
Let $F:(\widetilde{\Sigma},c)\to \mathbb{CP}^1$ be an equivariant holomorphic local diffeomorphism. Equivalently, $F$ is the developing map of a complex projective structure on $\Sigma$. For any other such map $\hat{F}$ on $(\widetilde{\Sigma},c)$, the osculation map $\Psi:\widetilde{\Sigma}\to \textrm{PSL}(2,\C)$ is the unique map determined by the property that, at every $z\in \widetilde{\Sigma}$, $\hat{F}(z)$ and $\Psi(z)(F(z))$ have the same holomorphic $2$-jet (see \cite{ElE} for more details).

Let $f=(f_1,\overline{f_2}):\widetilde{\Sigma}\to \mathbb{CP}^1\times \mathbb{CP}^1\backslash \Delta\simeq \mathbb{G}$ be an equivariant admissible immersion with pullback complex metric $g=f^*\inners$ and assume that $f_1$ is a local diffeomorphism. The map $f_1$ becomes holomorphic once we equip $\Sigma$ with the pullback complex structure $c_1$. Let $S(f_1)$ be the Schwarzian derivative of $f_1$. Let $f_1^q: \widetilde{\Sigma}\to \mathbb{CP}^1$ be another local diffeomorphism with the same pullback complex structure $c_1$ and such that $S(f_1^q)=S(f_1)+q$ for some holomorphic quadratic differential $q$. Denoting by $\Psi$ the osculation map for $f_1$ and $f_1^q$, define $\overline{f}_2^q=\Psi(\cdot)(\overline{f_2})$. It is proved in \cite[Section 4.2]{ElE} that the map $$f^q=(f_1^q,\overline{f}_2^q):\widetilde{\Sigma}\to \mathbb{CP}^1\times \mathbb{CP}^1\backslash \Delta$$ is an admissible immersion with pullback metric $g-2q$. 

If, in addition, $\overline{f_2}$ is an orientation-reversing local diffeomorphism, then $g$ is a positive complex metric inducing $(c_1, \overline{c_2})$, with $\overline{c_2}$ the pullback complex structure via $\overline{f_2}$. Observe that, in this case, by Proposition \ref{prop: hopf}, the map $f\colon (\widetilde \Sigma, c_1, \overline{c_2}) \to \mathbb G$ is complex harmonic with Hopf differentials $q_1=q$ and $\overline {q_2}=0$. If we want to see this in the context of Theorem \ref{thm: Bochner equation and harmonicity}, just note that, for pairs of Hopf differentials of the form $(q_1, \overline{q_2})=(-2q, 0)$, the function $U\equiv 1$ solves (\ref{eq: Bochner}), hence  $f^*{\inners}-2q$ is the pullback metric of a complex harmonic immersion.

Let's now assume $\Sigma=S$ closed, and let $g$ be the Bers metric $h=\hpair$. By definition of the Bers embedding relative to $\{[c_1]\}\times\mathcal T(\overline S)$, say $\overline{\mathcal Be}_{c_1}$, $f_1$ is the developing map of a complex projective structure with Schwarzian $$S(f_1)=\overline{\mathcal Be}_{c_1}([\overline{c_2}]).$$ Writing $f^q=(f_1^q,\overline{f_2}^q)$, since the pullback metric determines the admissible immersion, $f_1^q$ is the developing map of a complex projective structure with Schwarzian $$S(f_1^q)=\overline{\mathcal Be}_{c_1}([\overline{c_2}])+q.$$ Putting it all together, we can think of $f^q$, or equivalently $h-2q$, as the data of the development $f_1^q$ of a complex projective structure with underlying complex structure $c_1$ and Schwarzian determined by $\overline{c_2}$ and $q$, together with a ``transverse" map $\overline{f_2}^q$ to $\mathbb{CP}^1$ that's equivariant for the same holonomy and that makes the map $(f_1^q,\overline{f_2}^q)$ complex harmonic. Note that all of the discussion above goes through analogously if we instead work with and deform a local diffeomorphism $\overline{f_2}$, arriving at metrics of the form $h+\overline{q}$, where $q$ is $c_2$-holomorphic.

\subsubsection{Mappings into complex geodesics}
In contrast to the admissible complex harmonic maps produced from Theorem \ref{thm: Bochner equation and harmonicity}, here we consider the most degenerate complex harmonic maps. Degenerate harmonic maps into hyperbolic spaces are just maps onto geodesics. The right analogy in the complex world is maps onto complex geodesics. 

Recall that every complex geodesic is isometric to either $\C$ with the $0$-metric (that is, the metric $\inners$ is null on the geodesic), or to $(\C^*, \frac{dz^2}{z^2})$. 

We begin with copies of $\C$. Although $\inners$ degenerates, it still makes sense to discuss complex harmonic maps, since the equation depends only on the Levi-Civita connection of $\mathbb{G}$. Since we're mapping into a complex geodesic, the Levi-Civita connection restricts to the trivial one, i.e., the affine structure is the same as that of $(\C,dz^2).$ Thus, the complex harmonic map equation is the scalar complex harmonic map equation $$d(\Pi_+ df)=0,$$ or $\Delta_h f=0$, where $h$ a positive complex metric inducing the specified complex structures. Let $(c_1,\overline{c_2})$ be the complex structures on the source surface. Viewing $df$ as a $1$-form and decomposing $$df=\Pi_+ df + \Pi_- df,$$ the expressions (\ref{eq: harmonicity in coords 1}) and (\ref{eq: harmonicity in coords 2}) say that $\Pi_+ df$ is a $c_1$-holomorphic $1$-form, and $\Pi_- df$ is $\overline{c_2}$-holomorphic $1$-form. By integration, any complex harmonic map to a degenerate complex geodesic (or, to $(\C,dz^2)$) identifies with a function of the form $$u+\overline{v},$$ where $u$ is $c_1$-holomorphic and $\overline{v}$ is $\overline{c_2}$-holomorphic. If $c_1=c_2$, then $f$ is a real harmonic map, landing on a real geodesic, if and only if $u=v$.

For the non-degenerate geodesics, any complex harmonic map to $(\C^*,\frac{dz^2}{z^2})$ or a copy thereof locally lifts to a complex harmonic map to $(\C,dz^2)$. Since the exponential map $z\mapsto e^z$ intertwines the affine structures, upon identification, we locally can write any such map as $$\textrm{exp}(u+\overline{v}),$$ where $u$ is $c_1$-holomorphic and $\overline{v}$ is $\overline{c_2}$-holomorphic.

\subsubsection{Exotic complex harmonic maps}
Ordinary harmonic maps, and complex harmonic maps close to the real locus, are quite rigid. Although we don't pursue it deeply in this paper, there are lots of complex harmonic maps that are very ``far" from the real locus, and which can be quite wild.
To illustrate this phenomenon, we consider complex harmonic maps from $\C$ to $\mathbb{G}$, where we have explicit formulas. For ordinary harmonic maps from $\C$ to $\mathbb{H}^2$ and any polynomial quadratic differential $q$ on $\C$, there exists a unique harmonic diffeomorphism (onto its image), see \cite{WanAu1994}, \cite{LiVortexEquations}. There is no non-constant harmonic map from $\C\to \mathbb{H}^2$ with $q=0$, by the Liouville theorem for holomorphic functions. In contrast, we have the following.
\begin{thm}\label{thm: exotic harmonic maps}
    There exist infinite dimensional families of non-constant complex harmonic maps from $\C$ to $\mathbb{G}$ with Hopf differentials $(q_1,0)$, $(0,\overline{q_2})$, and $(0,0).$
\end{thm}
The theorem is more or less immediate from the existence of hyperbolic complex metrics on $\C$, which is independently interesting to note.
\begin{lem} 
    Let $f_1,f_2:\C\to \C$ be entire functions such that $f_1'(z)f_2'(\overline{z})$ and $(f_1(z)-f_2(\overline{z}))$ both have no zeros, and define $F=(f_1,f_2):\C\to \mathbb{CP}^1\times \mathbb{CP}^1\backslash\Delta$. Then $$F^*\inners = \frac{-4f_1'(z)f_2'(\overline{z})}{(f_1(z)-f_2(\overline{z}))^2}|dz|^2$$ defines a positive complex metric of constant curvature $-1$.
\end{lem}
\begin{proof}
By our non-vanishing assumptions, $F$ is an admissible immersion, hence $F^*\inners$ is a positive complex metric of constant curvature $-1$.
\end{proof}
\begin{proof}[Proof of Theorem \ref{thm: exotic harmonic maps}]
Any $F$ as above gives a complex harmonic map with Hopf differentials $(0,0)$. By \cite[Proposition 3.9]{ElE}, any metric of the form $F^*\inners + q_1$ or $F^*\inners +\overline{q_2}$ is non-degenerate with curvature $-1$, and hence by Proposition \ref{prop: metric to immersion} determines an admissible immersion to $\mathbb{G}$, which is harmonic by Proposition \ref{prop: hopf}.
\end{proof}
For an explicit example to feed into the results above, we can take $f_1(z)=e^z$ and $f_2(\overline{z}) = e^{\overline{z}}+i$. Then clearly both $f_1'$ and $f_2'$ never vanish, and $(f_1-f_2)(z) = 2e^{\textrm{Re}z}\sin (\textrm{Im} z)+i$, which also never vanishes.

\subsection{The $\C^*$-action and $\mathbb{CP}^1$ structures}\label{subsec: C* action}
There is a $\C^*$-action on the set of complex harmonic maps provided by Theorem \ref{thm: Bochner equation and harmonicity}: for $\zeta\in \C^*$, change $(q_1,\overline{q_2})$ to $(\zeta^2 q_1,\zeta^{-2}\overline{q_2})$ and leave $U$ as is. This gives a new harmonic map because the Bochner equation for $(q_1,\overline{q_2})$ is preserved. 
We use $\zeta^2$ instead of $\zeta$ to make this compatible with the $\C^*$-action on complex harmonic bundles under the bi-Hitchin map (see \cite[Section 4]{ES_Chmaps}). Here we give a geometric interpretation to the $\C^*$-action; this will be used in both Theorems \ref{thm: main} and \ref{thm: twistor lines}.

Let $c$ be a complex structure on $\Sigma$ with hyperbolic metric $h$. Given a complex harmonic map $f=(f_1,\overline{f_2}):(\widetilde{\Sigma},c,\overline{c})\to \mathbb{G}$ with Bochner equation data $(q_1, \overline{q_2}, U)$,
we define its \textbf{Beltrami forms} by $$\mu_1= \frac{q_1}{Uh}, \hspace{1mm} \overline{\mu_2}= \frac{\overline{q_2}}{Uh}.$$ 
When $f$ factors through a map to $\mathbb{H}^2$, $\mu_1$ and $\overline{\mu_2}$ are the ordinary Beltrami form and its conjugate.

By Theorem \ref{thm: Bochner equation and harmonicity} and \cite[Theorem 5.1]{BEE}, there is a unique holomorphic association $\zeta\mapsto f^\zeta=(f_1^\zeta,\overline{f_2}^\zeta)$, where $f^\zeta$ is a new complex harmonic map associated with data $(\zeta^2 q_1, \zeta^{-2}\overline{q_2}, U)$ and $f=f^1$. Presumably there is a version of the result below for $c_1\neq c_2$, but we don't pursue this here.
\begin{thm}\label{thm: A and B}
    Let $f=(f_1,\overline{f_2}):(\widetilde{\Sigma},c,\overline{c})\to \mathbb{G}$ be a complex harmonic map such that $g=f^*\inners$ is positive. Assume $f_1$ and $\overline {f_2}$ respectively preserve and invert the orientations as maps $\widetilde \Sigma\to \CP^1$. Define subsets $A,B\subset \C^*$ by $$A=\{\zeta \in \C^* : |\zeta|^2> ||\mu_2||_\infty\}, \hspace{1mm} B= \{\zeta \in \C^* : ||\mu_1||_\infty<|\zeta|^{-2}\}.$$ Then, for $\zeta \in A$, $f_1^\zeta$ is an orientation-preserving local diffeomorphism, and, for $\zeta\in B$, $\overline{f_2}^\zeta$ is an orientation reversing local diffeomorphism.
\end{thm}

\begin{proof}
Note that positivity guarantees both $f_1$ and $\overline{f_2}$ give pullback complex structures. The orientation assumption implies $|\mu_1|<1$ and $|\mu_2|<1$ (this is justified by an expression for $g^\zeta$ below with $\zeta=1$).

We can write the pullback metric $g^\zeta = (f^\zeta)^*\inners$ in two ways: 1) in terms of $f_1^\zeta$ and $\overline{f_2^\zeta}$ as $$g^\zeta = -\frac{4}{(f_1^\zeta-\overline{f_2}^\zeta)^2}df_1^\zeta d\overline{f_2}^\zeta$$ (here, $df_i^\zeta$ is interpreted as a complex valued $1$-form on $\Sigma$), and 2) in terms of the harmonic map data as $$g^\zeta = (Uh + \frac{q_1\overline{q_2}}{Uh})+\zeta^2 q_1 + \zeta^{-2}\overline{q_2}.$$
Let $z$ be a holomorphic local coordinate for $c$ and write $h=\lambda dz d\overline{z}$. Taking the second expression, we factor  
$$g^\zeta = U\lambda(dz+\zeta^{-2}\overline{\mu_2}d{z})(d\overline z+\zeta^2\mu_1d\overline{z}).$$ The expressions $dz+\zeta^{-2}\overline{\mu_2}d{z}$ and $d\overline z+\zeta^2\mu_1d\overline{z}$ are locally defined $1$-forms, which we denote by $\alpha^{\zeta}$ and $\beta^{\zeta}$ respectively. Their spans determine line sub-bundles of $\C T^*\Sigma$, which we denote by just $\textrm{Span}_{\C}\alpha^{\zeta}$ and $\textrm{Span}_{\C}\beta^{\zeta}$. Let $\ell_1^\zeta$ be the complex line in $\C T^*\Sigma$ spanned by $df_1^\zeta$ and let $\ell_2^\zeta$ be the complex line in $\C T^*\Sigma$ spanned by $d\overline{f_2}^\zeta.$ Comparing expressions for the pullback metric, we have that over every point in our local coordinate patch, as unordered sets $$\{\ell_1^\zeta,\ell_2^\zeta\}=\{\textrm{Span}_{\C}\alpha^{\zeta},\textrm{Span}_{\C}\beta^{\zeta}\}.$$

If $\zeta\in A$, then $\alpha^\zeta$ has no real vector in its kernel, so $\textrm{Span}_{\C}\alpha^\zeta$ never defines a point in $\mathbb{P}T^*S$. Note that for such $\zeta$, $\textrm{Span}_{\C}\alpha^{\zeta}$ lies in the same component of  $\mathbb{PC}T^*S\backslash \mathbb{P}T^*S$ as $[dz]$. Since $f_1$ is an orientation preserving local diffeomorphism, for $\zeta$ close to $1$, $df_1^\zeta$ continues to be an orientation preserving local diffeomorphism, and moreover $\ell_1^\zeta$ lies in the same component as $[dz]$. Thus, for $\zeta$ close to $1$, $\ell_1^\zeta=\textrm{Span}_{\C}\alpha^{\zeta}$. By holomorphicity, we deduce that $\ell_1^\zeta=\textrm{Span}_{\C}\alpha^{\zeta}$ for all $\zeta\in A$. Moreover, $df_1^\zeta$ has no real vector in its kernel and gives the same component as $[dz]$, so it is an orientation preserving local diffeomorphism.

The exact same argument applies to $\beta^\zeta$ to shows that for $\zeta \in B$, $\ell_2^\zeta=\textrm{Span}_{\C}\beta^{\zeta}$, and $\overline{f_2}^\zeta$ is an orientation reversing local diffeomorphism.
\end{proof}
A consequence of Theorem \ref{thm: A and B} is that for $\zeta \in A$, $f_1^\zeta$ determines a $\mathbb{CP}^1$-structure with underlying complex structure determined by the Beltrami form $\zeta^{-2}\overline{\mu_2}$, and for $\zeta \in B$, $\overline{f_2}^\zeta$ determines a $\mathbb{CP}^1$-structure with underlying complex structure determined by $\zeta^2\mu_1$. When $\zeta \in A\cap B$, if we take the pair of complex structures corresponding to $(\zeta^{-2}\overline{\mu_2}, \zeta^2\mu_1)$, then $f^\zeta$ is complex harmonic with vanishing Hopf differentials. 

The main application toward Theorem \ref{thm: main} comes from the fact that complex projective structures are known to have smooth holonomy. We thus have the following.
\begin{cor}
\label{cor: C* preserves smooth cv}
    Let $\Sigma$ be a closed surface and let $f=(f_1,f_2):(\widetilde{\Sigma},c,\overline{c})\to \mathbb{G}$ be a complex harmonic map such that $g=f^*\inners$ is positive. For $\zeta\in \C^*$, the complex harmonic map $f^\zeta$ has holonomy in the smooth locus of the $\textrm{SL}(2,\C)$ character variety.
\end{cor}

\section{The unifying map}\label{sec: main theorem}

\subsection{Spaces of harmonic maps}

From now on, we restrict to the special case of $\Sigma=S$ a closed oriented surface of genus $\mathtt g\ge 2$ and $c_1={c}_2=:c$. We denote by $h$ the Riemannian hyperbolic metric conformal to $c$. Also, we restrict to solutions $U$ of Equation \eqref{eq: Bochner} that admit a logarithm on $S$, hence to functions of the form $U=e^{2u}$ with $u\colon S\to \mathbb C$ a solution to the equation
\begin{equation}
\label{eq: logBochner}
    \Delta_h u= e^{2u}-\frac{q_1\overline{q_2} }{h^2}e^{-2u}-1.
\end{equation}
We are therefore interested in the space
\[
\mathcal H(S,c):=\left\{(q_1, \overline{q_2}, u)\in H^0(\mathcal K_c^2)\times H^0(\mathcal K_{\overline c}^2)\times C^{\infty}(S,\C)\ |\ 
\Delta_h u= e^{2u}-\frac{q_1\overline{q_2} }{h^2}e^{-2u}-1
\right\},
\]
which parametrizes a set of equivariant complex harmonic maps from $(\widetilde{S},c)$ to $\mathbb G$. We set
\begin{equation}
\label{eq: vb g}
\vb*g (q_1, \overline{q_2}, u)= q_1 +\left(e^{2u} + e^{-2u}\frac{q_1\overline{q_2}}{h^2 }\right)h + \overline{q_2}.
\end{equation}
to be the pullback metric.

We also define the holonomy map $$\mathrm{hol}: \mathcal H(S,c)\to \chi(\pi_1(S), \PSL(2,\C))$$ which takes 
$(q_1, \overline{q_2}, u)$ to the holonomy of the complex harmonic map $(\widetilde S,c)\to \mathbb G$ with pull-back metric $\vb*g (q_1, \overline{q_2}, u)$. Above, $\chi(\pi_1(S), \PSL(2,\C))$ is the Hausdorffification of the quotient of $\textrm{Hom}(\pi_1(S),\textrm{PSL}(2,\C))$ by the conjugation action of $\textrm{PSL}(2,\C)$. We view the analytic character variety $\chi^{an}(\pi_1(S), \textrm{PSL}(2,\C))$ as a distinguished subset.

For all $(q_1, \overline{q_2}, u)\in \mathcal H(S,c)$, the corresponding Bochner equation linearizes to 
\begin{equation}
    \label{eq: linearized bochner}
L_{(q_1, \overline{q_2}, u)}\colon \mathcal C^{\infty}(S)\to\mathcal C^{\infty}(S), \qquad L_{(q_1, \overline{q_2}, u)}(\dot{u})= \Delta_h \dot u-   \left(2 e^{2u}  - 2e^{-2u}\frac{q_1\overline{q_2}}{h^2}\right)\dot u.
\end{equation}
Since $L_{(q_1, \overline{q_2}, u)}$ is connected to $\Delta_h$, it is Fredholm of index zero. By the holomorphic implicit function theorem,
\[
\mathcal H^*(S,c):= \{(q_1, \overline{q_2}, u)\in \mathcal H(S,c)\ |\ \mathrm{hol}(q_1, \overline{q_2},u)\in \chi^{an}(\pi_1(S), \textrm{PSL}(2\C)), \text{ and } L_{(q_1, \overline{q_2}, u)}\ \text{ is bijective} \}
\]
is a complex manifold of dimension $6g-6$ such that the forgetful map

\begin{equation}
\label{eq: projection pi}
\pi\colon \mathcal H^*(S,c)\to H^0(\mathcal K^2_c)\times H^0(\mathcal K^2_{\overline c}), \qquad
    (q_1, \overline{q_2}, u)\mapsto (q_1,\overline{q_2})
\end{equation}
is a local biholomorphism. Here, $\chi^{an}(\pi_1(S), \C)$ is the analytic character variety, the class of representations such that the conjugation quotient yields a complex manifold (the bijectivity condition gives a complex manifold, and the $\chi^{an}$ condition restricts us to an open subset).

\begin{prop}
\label{prop: hol is holomorphic}
    The holonomy map $\mathrm{hol}\colon\mathcal H^*(S,c)\to \chi^{an}(\pi_1(S),\PSL(2,\C))$ is holomorphic.
\end{prop}

We will prove it using the following fact from \cite[Theorem D]{BEE}.
\begin{lem}
\label{lem: holo dependence Bonsante}
    Let $Z$ be a complex manifold, let $\{g_{\xi}\}_{\xi\in Z}$ be a family of complex metrics of constant curvature $-1$ and with holonomy in $\chi^{an}(\pi_1(S), \PSL(2,\C))$. Assume that, for all $x\in S$, the map $Z\to \mathrm{Sym}_2(\C T_x S)$ given by $\xi\mapsto {g_{\xi}}_{|x}$ is holomorphic. Then, the map $Z\to \chi^{an}(\pi_1(S),\PSL(2,\C))$, $\xi \mapsto\mathrm{hol}(g_{\xi})$, is holomorphic. 
\end{lem}

\begin{proof}[Proof of Proposition \ref{prop: hol is holomorphic}]
 Recalling that the projection $\mathcal H^*(S,c) \to H^0(\mathcal K^2_c)\times H^0(\mathcal K^2_{\overline c})$ is a local biholomorphism, consider any local inverse $U\to V$, $(q_1, \overline{q_2})\mapsto (q_1, \overline{q_2}, u(q_1, \overline{q_2}))$. Since the projection from $\mathcal H^*(S,c)$ to each of its factors is holomorphic, for all $x\in S$, the evaluation maps $(q_1, \overline{q_2}, u)\mapsto q_1 |_x, \overline{q_2}|_x, u(x) $ are holomorphic.
 We deduce from the formula (\ref{eq: logBochner}) that the evaluation of $\vb*g(q_1, \overline{q_2}, u)$ at each $x$ depends holomorphically on $(q_1, \overline{q_2}, u)$. We conclude by Lemma \ref{lem: holo dependence Bonsante}.
\end{proof}
Finally, we put the $\C^*$-action on $\mathcal{H}(S,c)$. 
\begin{defn}
    The $\C^*$-action on $\mathcal{H}(S,c)$ is given by $\zeta\cdot (q_1,\overline{q_2},u)=(\zeta^2 q_1,\zeta^{-2}\overline{q_2},u).$
\end{defn}
About $\mathcal{H}^*(S,c)$, it is immediate from the formula \eqref{eq: linearized bochner} that the $\C^*$-action preserves the bijectivity (or lack thereof) of $L_{(q_1,\overline{q_2},u)}$, but it is not clear that the action preserves the analytic holonomy property.

\subsection{Key loci in $\mathcal H^*(S,c)$}\label{subsec: key loci}
\subsubsection{Harmonic maps between hyperbolic surfaces}

As discussed previously, by embedding $\mathbb{H}^2\to \mathbb{G}$, ordinary harmonic maps become complex harmonic maps. Choosing the specific totally geodesic isometric embedding $z\mapsto (z,\overline{z})$, Fuchsian harmonic maps into $\mathbb{H}^2$ correspond to the locus \[
\mathbf W(S,c):= \{(q,\overline q, u)\in \mathcal H(S,c)\ |\ u\text{ is real-valued}\}
\] in $\mathcal{H}(S,c)$. By a simple maximum principle argument, $\mathbf W(S,c)\subset \mathcal H^*(S,c)$.
As in the introduction, we have Hitchin-Wolf's parametrization $\mathcal W\colon H^0(\mathcal K_c^2)\xrightarrow{\sim}  \mathcal T(S)$, and for all  $(q,\overline q, u)\in \mathbf W(S,c)$, $\mathrm{hol}(q,\overline q, u)= (\mathcal W(q), \overline{\mathcal W(q)})\in \TSTS\cong\QF(S)\subset \chi(\pi_1(S), \PSL(2,\C))$.

\subsubsection{Complex projective structures}
We define the properly embedded submanifolds 
\begin{align*}
    \mathbf S(S,c):=\{ (q, 0, 0)\ |\ q\in H^0(\mathcal K_c^2)\}\subset \mathcal H(S,c)\\
   \overline{\mathbf S}(S,c):=\{ (0, \overline q, 0)\ |\ q\in H^0(\mathcal K_c^2)\}\subset \mathcal H(S,c).
\end{align*}

Since $(0,0,0)\in \mathcal{H}^*(S,c)$ and $\mathcal{H}^*(S,c)$ is open, $\C^*$-invariance and shows that the linearization of the Bochner equation is an isomorphism on $\mathbf S(S,c)$ and $\overline{\mathbf S}(S,c)$. Since complex projective structures have smooth holonomy, we conclude that both $\mathbf S(S,c)$ and $\overline{\mathbf S}(S,c)$ lie in $\mathcal{H}^*(S,c)$. By the discussion in Section \ref{subsubsection: CP1 structures}, $\mathbf S(S,c)$ and $\overline{\mathbf S}(S,c)$ parametrize $\mathbb{CP}^1$-structures for $(S,c)$ and $(S,\overline c)$ respectively, and \begin{align*}
  \mathrm{hol}(q,0,0)&= \mathcal S\left(-\frac 12 q\right)\\
  \mathrm{hol}(0,\overline q, 0)&=\overline {\mathcal S}\left(-\frac 12 \overline q\right),
\end{align*}
where $\mathcal{S}$ is the Schwarzian parametrization from the introduction and $\overline{\mathcal{S}}$ is its conjugate.

Observe that $(q,0,0)\in \mathbf S(S,c)$ corresponds to a $\CP^1$-structure for $(S,c)$ coming from the Bers embedding if and only if $q\in \overline{\mathcal Be}(\mathcal T(\overline S))\subset H^0(\mathcal K_c^2)$. Similar for elements $(0, \overline q, 0)\in \overline{\mathbf S}(S,c)$.
In the identification $\QF(S)\cong \TSTS$ provided by Bers' theorem, we can summarize the interplay between these spaces through the following diagram.
\begin{equation}
\begin{tikzcd}
& H^0(\mathcal K_c^2)\times\{0\} & \arrow[l, "\sim" , "\pi" swap] \mathbf{S}(S,c)\arrow[r, "hol" ] & \chi^{an}(\pi_1(S), \PSL(2,\C))&\\
& \overline {\mathcal Be}(\mathcal T(\overline S))\times\{0\} \arrow[u, phantom, "\subset" rotate=90]  & \arrow[l, "\sim" , "\pi" swap] \{(q, 0, 0)\ |\  q\in   \overline {\mathcal Be}(\mathcal T(\overline S))\} \arrow[r, "hol" , "\sim" swap]  \arrow[u, phantom, "\subset" rotate=90]&\{[c]\}\times\mathcal T(\overline S) \subset \QF(S)\arrow[u, phantom, "\subset" rotate=90]&
\end{tikzcd}
\end{equation}

\subsubsection{Almost-Fuchsian immersions}\label{subsubsec: almost fuchsian}
By the discussion in Section \ref{subsubsec: mimimal immersions}, an element of the form $(q, -\overline q, u)\in \mathcal H(S,c)$ with $u$ real-valued determines a minimal equivariant immersion $\widetilde S\to \mathbb H^3$ with fundamental forms $\Is= e^{2u}h$ and $\II=\frac 12 (q+\overline q)$. Such an immersion is said to be \textbf{almost-Fuchsian} if the eigenvalues of the shape operator $\Is^{-1}\II$ lie between $-1$ and $+1$; equivalently, a minimal immersion is almost-Fuchsian if $det_{\Is}\II > -1$, i.e., if
\begin{equation}
\label{eq: almost Fuchsian condition}
  \frac 1 {e^{4u}}  \frac{|q|^2}{h^2}<1. 
\end{equation}
The space of almost-Fuchsian immersions up to ambient isometries is therefore parametrized by
\[
\mathbf{U}(S,c):=\{ (q, -\overline q, u)\ |\ \text{$u$ is real, and } \frac{|q|^2}{h^2} e^{-4u}<1 \}.
\]
By a simple maximum principle argument applied to Equation \eqref{eq: logBochner}, we see that for $q\in H^0(\mathcal K_c^2)$ there exists at most one real-valued function $u$ such that $(q, -\overline q, u)\in \mathbf{U}(S,c)$. Applying again the maximum principle to $L_{(q, -\overline q, u)}\dot u=0$, we conclude that $\mathbf U(S,c)\subset \mathcal H^*(S, c)$ is an embedded submanifold, hence the projection on the first factor $\mathbf U(S,c)\to H^0(\mathcal K_c^2)$ is a smooth open embedding, whose image we denote by $\mathsf{Uh}(S,c)$. 

Almost-Fuchsian representations are defined to be the holonomies of almost-Fuchsian immersions. As proved by Uhlenbeck \cite{Uh}, for every almost-Fuchsian representation $\rho$  there exists a unique $\rho$-equivariant almost-Fuchsian immersion (up to precomposing with lifts of isotopies of $S$), and almost-Fuchsian representations determine an open subset $\mathrm{AF}(S)$ of $\QF(S)\subset \chi^{an}(\pi_1(S), \PSL(2,\mathbb C))$. 
We define $\mathrm{AF}(S,c)\subset\mathrm{AF}(S)$ as the set of conjugacy classes of holonomies of almost-Fuchsian immersions with first fundamental form conformal to $c$. Putting it all together, we have that 
\begin{align*}
    H^0(\mathcal K_c^2)\supset  \mathsf{Uh} (S,c) \ \  \xleftarrow{\sim}\ \ &\mathbf U(S, c)\ \ \xrightarrow{\sim}\ \ \mathrm{AF}(S,c)\subset \mathrm{AF}(S)\\
    q\mapsfrom\ \ (&q,-\overline q, u) \ \ \mapsto \mathrm{hol}( q,-\overline q, u )
\end{align*}

\subsection{The holonomy map and the $\C^*$-action}\label{subsec: holonomy immersion}

For $(q_1,\overline{q_2}, u)\in \mathcal H^*(S,c)$ recall the notation $\vb*g(q_1, \overline{q_2}, u)$ (Equation \eqref{eq: vb g}). We prove the following.
\begin{thm}\label{thm: immersion}
    Let $U\subset \mathcal H^*(S,c)$ be a connected open subset such that: \begin{itemize}
    \item $hol$ is a local biholomorphism on $U$;
        \item $\vb*g(q_1, \overline{q_2}, u)$ is a positive complex metric for all $(q_1, \overline{q_2}, u)\in U$;
    \item $U$ contains $\mathbf W(S,c)$. 
    \end{itemize}
    Then $\mathbb C^*\cdot U$ is contained in $\mathcal H^*(S,c)$ and $\mathrm{hol}$ is a local biholomorphism on $\mathbb C^*\cdot U$.
\end{thm}
The containment $\C^*\cdot U\subset \mathcal{H}^*(S,c)$ is just a consequence of Corollary \ref{cor: C* preserves smooth cv}. To prove the local biholomorphism property, we use the (complex) Goldman symplectic form $\omega^{\C}$ on $\chi^{an}(\pi_1(S), \PSL(2,\C))$; see \cite{Hitchin:1986vp} for one definition (which we won't make use of). Since the dimensions of both spaces coincide, it suffices to show that $\mathrm{hol}$ is an immersion. We use $\omega$ to denote the real Goldman symplectic form on the (analytic) $\textrm{PSL}(2,\R)$ character variety, which on Teichm{\"u}ller space is the Weil-Petersson symplectic form. Then, $\omega^{\C}$ is its (global) holomorphic extension to $\chi^{an}(\pi_1(S), \PSL(2,\C))$. If $\mathrm{hol}^*\omega^{\C}$ is a non-degenerate symplectic form on $U'$, then $\mathrm{hol}$ is an immersion on $U'$; since $\omega^{\C}$ is non-degenerate, the converse is true too. Thus, Theorem \ref{thm: immersion} is an immediate consequence of the lemma below.
\begin{lem}\label{lem: symplectic action}
    On the (pre-)symplectic manifold $(\C^*\cdot U,\mathrm{hol}^*\omega^{\C})$, the $\mathbb C^*$-action is (pre-)symplectic.
\end{lem}
\begin{proof}
    The $\C^*$-action restricts to an $S^1$-action on the real submanifold of real harmonic maps $\mathbf W(S,c)$. Using the above identification between $\mathbf W(S,c)$ and $H^0(\mathcal{K}_c^2)$, it follows from \cite[pp. 92 and Proposition 11.17]{Hitchin:1986vp} that on $\mathbf W(S,c)$, the $S^1$-action is symplectic for $\omega$. Indeed, in \cite{Hitchin:1986vp}, Hitchin's section identifies $H^0(\mathcal{K}_c^2)$ with the Hitchin component of the moduli space of $\textrm{SL}(2,\R)$-Higgs bundles, on which the pullback of the Goldman form is identified with his symplectic form $``\omega_1$" (Proposition 11.17), which he shows to be invariant under the $S^1$-action (pp. 92). The result now follows by analytic continuation: the action $\C^*\times \C^*\cdot U\to \C^*\cdot U$ holomorphically extends the action $S^1\times \mathbf W(S,c)\to \mathbf W(S,c)$.
\end{proof}

\begin{remark}
    We omit the argument since we don't need the result, but Lemma \ref{lem: symplectic action} holds even if we drop the condition that $U$ contains $\mathbf{W}(S,c)$, thus forbidding an analytic continuation argument. It follows from a routine calculation, using the explicit formula for $\omega^{\C}$ in terms of complex metrics from \cite{ElE} (i.e., the extension of Fischer-Tromba's formula for the Weil-Petersson metric) and the expression (\ref{eq: pullback, complex}). Restricting to the real locus, this also gives an alternative argument for the $S^1$-invariance.
\end{remark}

\subsection{The unifying maps}

\begin{proof} [Proof of Theorem \ref{thm: main}]

    The main goal is to find an open connected subset $\Omega\subset H^0(\mathcal K_c^2)\times H^0(\mathcal K_{\overline c}^2) $ that is $\mathbb C^*$-invariant, such that $\mathrm{hol}|_\Omega$ is a local biholomorphism, and on which $\mathrm{hol}|_\Omega$ admits an inverse $s\colon \Omega\to \mathcal H^*(S,c)$ of $\pi$. Then, the desired mapping is just $\Phi:=\mathrm{hol}\circ s$. As commented after the theorem statement, uniqueness is automatic.

    First of all, we prove that for every point of $H^0(\mathcal K_c^2)\times H^0(\mathcal K_{\overline c}^2)$, every neighborhood contains a pre-compact neighborhood such that the intersection with every $\mathbb C^*$-orbit is either empty or connected; we will call such neighborhoods ``$\mathbb C^*$-good." This is trivial at the point $(0,0)\in \mathcal H(S,c)$ by taking neighborhoods of the form 
    \begin{equation}
    \label{eq: nbhd (0,0)}
        \{(q_1, \overline{q_2})\ |\ \|q_1\|<\eps_1, \|q_2\|<\eps_2\},
    \end{equation}
    where $\|\cdot \|$ denotes just any real norm on the spaces of holomorphic differentials. On the other hand, $\mathbb C^*$ acts freely (and with closed embedded orbits) on $H^0(\mathcal K_c^2)\times H^0(\mathcal K_{\overline c}^2)\setminus\{(0,0)\}$. Therefore, the infinitesimal orbit distribution is integrable (with orbits as integral curves), and by the Frobenius Theorem the foliation by $\mathbb C^*$-orbits is locally trivializable. In trivializations, it isn now immediate to construct neighborhoods such that the intersection with every orbit is connected. 

Now, observe that if $s\colon U\to \mathcal H^*(S,c)$ is an inverse of $\pi$, $U$ is $\mathbb C^*$-good, and both $(q_1,\overline{q_2})$ and $\zeta\cdot (q_1,\overline{q_2})$ lie in $U$, then $s(\zeta\cdot (q_1,\overline{q_2}))=\zeta\cdot s((q_1,\overline{q_2}))$. Indeed, for $\zeta$ sufficiently close to $1$,  $\zeta \cdot(q_1,\overline{q_2})= \pi (\zeta\cdot s((q_1,\overline{q_2}))) $ lies in $U$, and hence $s(\zeta\cdot (q_1,\overline{q_2}))=\zeta\cdot s((q_1,\overline{q_2}))$ is obtained by applying $s$ to both terms. This implies that the subset $$\{\zeta\in \mathbb C^* \ | \ \zeta\cdot (q_1,\overline{q_2})\in U \text{ and } s(\zeta\cdot (q_1,\overline{q_2}))=\zeta\cdot s((q_1,\overline{q_2})) \}$$
is open and closed in $\{\zeta\in \mathbb C^* \ | \ \zeta\cdot (q_1,\overline{q_2})\in U\}$, and hence it's the whole thing by the fact that $U$ is $\mathbb C^*$-good. As a consequence, if $s\colon U\to \mathcal H^*(S,c)$ is an inverse defined on a $\mathbb C^*$-good subset $U$, and if $\mathbb C^*\cdot U\subset \mathcal H^*(S,c)$, then $s$ extends uniquely to a biholomorphic inverse of $\pi$ on $\mathbb C^*\cdot U$. Indeed, take $s(\zeta\cdot (q_1,\overline{q_2}) ) := \zeta \cdot s((q_1,\overline{q_2}) )$ for all $ (q_1,\overline{q_2}) )\in U$.

Let $Z_1=\mathbf W(S,c)$ and $Z_2=\mathbf {U}(S,c)$.
By the Generalized Inverse Function Theorem (see Exercise 1.8.14 in \cite{GuillPoll}), there exists a global inverse $s\colon \Omega_0\to \mathcal H(S,c)$ with $s(\Omega_0)$ containing both $Z_1$ and $Z_2$. As discussed in Section \ref{subsec: key loci}, $\mathrm{hol}$ is a local biholomorphism at every point of $Z_1\cup Z_2$, hence we can assume that $\mathrm{hol}$ restricts to a local biholomorphism on $s(\Omega_0)$ and that $\vb*g(q_1, \overline{q_2}, u)$ is a positive metric for all $(q_1, \overline{q_2}, u)\in s(\Omega_0)$. 
We claim that, after possibly shrinking $\Omega_0$ further, $s$ extends $\mathbb C^*$-equivariantly to $\mathbb C^*\cdot \Omega_0$. Granting this, on $\mathbb C^*\cdot \Omega_0$, $s$ determines a global biholomorphic inverse of $\pi$. By Theorem \ref{thm: immersion}, $\mathrm{hol}$ is a local biholomorphism on $s(\Omega_0)$, hence $\Omega=\mathbb C^*\cdot\Omega_0$ and $\Phi=\mathrm{hol}\circ s$ give the theorem statement.

To realize the claim, we first extend $s$ to a $\mathbb C^*$-invariant neighborhood of $Z_1\cup Z_2\setminus\{(0,0)\}$. One can easily observe that the $\mathbb C^*$-action restricts to 
\[
Q=(H^0(S, \mathcal K_c^2)\setminus \{0\}) \times (H^0( \mathcal K_{\overline c}^2)\setminus \{0\})
\]
and that the action is free and proper over it: $Q/\mathbb C^*$ is therefore a complex manifold, and it contains $\pi(Z_1)/S^1$ and $\pi(Z_2)/S^1$ as disjoint closed embedded submanifolds. Now, cover each of these submanifolds with images of $\mathbb C^*$-good open subsets of $\Omega_0$ around points in $\pi(Z_k)$; refine these coverings with geodesically convex open subsets for some Riemannian metric on $Q/\mathbb C^*$: in this way, any two of them that intersect have convex intersection that intersects $\pi(Z_k)/S^1$. As a result, we obtain a covering of each $Z_k$ with a collection of $\mathbb C^*$-invariant open subsets over which $s$ can be extended to be $\mathbb C^*$-equivariant. Moreover, if any two of these open subsets intersect, then their intersection would be connected, and they would intersect on an open connected subset of $Z_k$ over which the inverses match. They therefore coincide on the whole intersection. We have hence constructed a $\mathbb C^*$-equivariant inverse $s$ on a $\mathbb C^*$-invariant neighborhood $\Omega_1$ of $\pi(Z_1)\cup\pi( Z_2)\setminus\{(0,0)\}$.
Up to shrinking the initial $\Omega_0$ to the union of some suitable $\mathbb C^*$-good open neighborhoods of the points of $\pi(Z_1)\cup\pi(Z_2)$, we can assume that these two inverses coincide on $\Omega_0\cap \Omega_1$. 

Finally, let $U_{0,0}$ be a $\mathbb C^*$-good neighborhood of $(0,0)$ of the form \eqref{eq: nbhd (0,0)} contained in $\Omega_0$. Observe that $s$ clearly glues to $U_{0,0}\cup \Omega_1$. We are left to show that $s$ extends $\mathbb C^*$-equivariantly to $\Omega\coloneqq\mathbb C^*\cdot U_{0,0}\cup \Omega_1$. This is simple to check: if $w\in \mathbb C^*\cdot U_{0,0}\cup \Omega_1$, then there exists $\zeta\in \mathbb C^*$ such that $\zeta\cdot w\in  U_{0,0}\cup \Omega_1$, thus
\[
s|_{\mathbb C^* \cdot U_{0,0}} (w)= \zeta^{-1}\cdot s(\zeta \cdot  w)= s|_{\Omega_1} (w).
\]
This completes the proof of the claim, and hence we're done.
\end{proof}

The deduction of Corollary \ref{cor: main cor} will follow by showing that, after restricting $\Omega$, $\Phi$ is a biholomorphism. To see this, we need the following lemma about the locus $\mathrm{AF}(S,c)$ in the identification $\mathrm{QF}(S)\cong \mathcal T(S)\times\mathcal T(\overline S)$.

\begin{lem}
\label{lem: global inj of hol}
    In the identification $\QF(S)\cong\mathcal T(S)\times\mathcal T(\overline S)$, for $(q, -\overline q, u)\in \mathbf U(S,c)$, denote  $\mathrm{hol}(q, -\overline q, u)=:([c^q_1],[ \overline c^q_2])$. Then,
    $$
    [c]=[c^q_1]\  \Longleftrightarrow\ [\overline c]=[\overline c^q_2] \ \Longleftrightarrow \ [c^q_1]=[ c^q_2]\ \Longleftrightarrow q=0.
    $$
\end{lem}

\begin{proof}
Clearly, if $q=0$, then all the other statements follow. If $[c^q_1]=[c^q_2]$, then the holonomy is Fuchsian, so the unique almost-Fuchsian immersion with that holonomy is the totally geodesic one, i.e., $q=0$.

Now, let $\Is$, $\II$, and $\mathtt B$ be the fundamental forms and the shape operator of the almost-Fuchsian immersion corresponding to $(q, -\overline q, u)$. By the work of Krasnov-Schlenker \cite{KrasnovSchlenker2008}, $c^{q}_1$ and $c^q_2$ give conformal structures (up to isotopy) for the Riemannian metrics $\Is_+^* = \Is((\mathrm{Id}+ \mathtt B)\cdot, (\mathrm{Id}+ \mathtt B)\cdot)$ and $\Is_-^*=\Is((\mathrm{Id}- \mathtt B)\cdot, (\mathrm{Id}- \mathtt B)\cdot)$ respectively. Observe that $\Is+\Is(\mathtt B\cdot, \mathtt B\cdot)$ is conformal to $\Is$, and that
 \[
 \Is_+^*= \Is+\Is(\mathtt B\cdot, \mathtt B\cdot) + q+\overline q \qquad \qquad \Is_-^*= \Is+\Is(\mathtt B\cdot, \mathtt B\cdot) - q-\overline q.
 \]
 By holomorphicity of the Hopf differentials, the identity maps $id\colon (S, \Is)\to (S, \Is_+^*)$ and $id\colon (S, \Is)\to (S, \Is_-^*)$ are harmonic. 

  Assuming that $c$ is conformal to $c^q_1$, we prove that $q=0$. The proof will follow in the analogous way for $c^{q}_2$. If $c=\phi^*(c^q_1)$ for $\phi$ isotopic to the identity, then $\phi\colon (S,c)\to (S, \Is^*_+)$ is harmonic and conformal (in fact, $\phi\colon (S,\phi^*(\Is^*_+)) \to (S, \Is^*_+) $ is an isometry). By construction, the curvature of $\Is_+^*$ is $-1+ \frac 2 {det(\mathtt B)}<0$, hence the harmonic map $(S,c)\to (S, \Is^*_+)$ isotopic to the identity is unique. Thus, $\phi=id$ and $id$ is conformal, hence $q=0$. 
\end{proof}

\begin{prop}
\label{prop: global inj of hol}
    The restriction of $\Phi$ on $$\pi(\mathbf W(S,c))\cup \pi(\mathbf U(S,c))\cup  \pi(\mathbf{Be}(S,c)) \cup \pi(\overline{\mathbf{Be}}(S,c) ) $$ is injective.
\end{prop}
\begin{proof}
By construction $\Phi$ is injective when restricted individually to each of the terms of the union. 
    By Lemma \ref{lem: global inj of hol}, their images through $\Phi$ pairwise meet only at $\Phi(0,0)=([c],[\overline c])\in \TSTS\cong \QF(S)$, hence the injectivity holds.
\end{proof}

We are now ready to prove Corollary \ref{cor: main cor}.
\begin{proof}[Proof of Corollary \ref{cor: main cor}]
  As mentioned before, to prove the statement, it is sufficient to prove that $\Phi$ is injective on an open subset $\Omega_3\subset H^0(\mathcal K_c^2)\times H^9(\mathcal K_{\overline c}^2)$ with $$ \Omega_3\supset  Z:= \pi(\mathbf W(S,c))\cup \pi(\mathbf {Be}(S,c))\cup\pi( \overline{\mathbf {Be}}(S,c))\cup\pi( \mathbf U(S,c)).$$  By Theorem \ref{thm: main}, $\Phi$ is a local biholomorphism at every point of $Z$ and, by Proposition \ref{prop: global inj of hol}, $\Phi$ is injective on $Z$.

By the Generalized Implicit Function Theorem (see \cite{GuillPoll}), $\Phi$ restricts to open embeddings around some closed neighbourhoods 
   \[
   \overline {\Omega^1_3}\supset \pi( \mathbf W(S,c)),\quad 
  \overline {\Omega_3^2}\supset \pi(\mathbf {Be}(S,c)),\quad
   \overline {\Omega_3^3}\supset \pi(\overline{\mathbf {Be}}(S, c)),\quad
  \overline {\Omega_3^4}\supset \pi( \mathbf U(S,c)).  
   \]
By defining 
\[
\mathcal A=int \left( \left\{\rho\in \bigcup_j \Phi(\Omega_3^j)\ |\ \text{if $\rho\in \Phi( \Omega_3^j) \cap  \Phi(\Omega_3^k)$, then } (\Phi|_{\Omega_3^j})^{-1}(\rho)= (\Phi|_{\Omega_3^k})^{-1}(\rho) \right\} \right) 
\]
we get that $\Phi$ restricts to a biholomorphism $\Phi\colon \Omega_3'\to \mathcal A$, where $\Omega_3'$ is an open subset of $\mathcal H^*(S,c)$ containing ${Z}$ (see \cite[Section 6.6]{ES_Chmaps} for a detailed argument of this type in a more complicated framework). By denoting $\Psi= (\Phi|_{\Omega'_3})^{-1}$, we get the statement of the corollary.
\end{proof}

\section{Higgs bundles}\label{sec: Higgs bundles}
Here we recall the necessary background on Higgs bundles and complex harmonic bundles, and then prove Theorem \ref{thm: twistor lines}. Throughout, fix a surface $\Sigma$ with complex structures $(c_1,\overline{c_2}).$ As in Section \ref{sec: definition}, we set $\Pi_+$ and $\Pi_-$ to be the projections for the splitting $\C T^*\Sigma = (T^*\Sigma)_{c_1}^{1,0}\oplus (T^*\Sigma)_{c_2}^{0,1}.$
\subsection{Complex harmonic bundles}
We recall the notion of complex harmonic bundles, as introduced and studied in \cite{ES_Chmaps}. In \cite{ES_Chmaps}, we worked with complex harmonic $G$-bundles, i.e., the principal bundle analog of complex harmonic bundles, and obtained vector bundles from $G$-bundles by taking representations and associated bundles. Here we work in the vector bundle context. Since this isn't the main focus of the current paper, we're a bit terse with our exposition; for basics on vector bundles and equivariant maps that we assume of the reader, see \cite[Section 4]{ES_Chmaps} and \cite{Li}.
\begin{defn}
    A \textbf{complex harmonic bundle} over $(\Sigma,c_1,\overline{c_2})$ is a tuple $(E,Q,\nabla, \phi_1,\overline{\phi_2})$, where 
    \begin{enumerate}
        \item $E$ is a complex vector bundle,
        \item $Q$ is a non-degenerate bilinear form on $E$,
        \item $\nabla$ is $Q$-compatible connection on $E$,
        \item and $\phi_1\in \Omega_{c_1}^{1,0}(\textrm{End}_0(E))$ and $\overline{\phi_2}\in \Omega_{c_2}^{0,1}(\textrm{End}_0(E))$ are $Q$-symmetric sections satisfying $\nabla\phi_1=\nabla\overline{\phi_2}=0,$
    \end{enumerate}
  and such that the connection $$D:=\nabla + \phi_1 +\overline{\phi_2}$$ is flat.
\end{defn}

We consider the complex manifold $\mathbb{X}=\textrm{SL}(n,\C)/\textrm{SO}(n,\C)$, identified with the space of symmetric matrices in $\textrm{SL}(n,\C)$. The Killing form on the Lie algebra determines a holomorphic Riemannian metric on $\mathbb{X}$. In our model, at a matrix $A$ this is $(X,Y)\mapsto \frac{n}{2}\textrm{tr}(A^{-1}XA^{-1}Y).$ 

For more details on all of the discussion below, see \cite[Section 4.1]{ES_Chmaps}.
A complex harmonic map $f$ to $\mathbb{X}=\textrm{SL}(n,\C)/\textrm{SO}(n,\C)$, equivariant for some representation $\rho:\pi_1(\Sigma)\to \textrm{SL}(n,\C)$, produces a complex harmonic bundle as follows. We begin with the associated bundle $$E_\rho=\widetilde{\Sigma}\times_\rho \C^n = \{(x,v)\in \widetilde{\Sigma}\times \C^n\}/\sim,$$ where $\sim$ identifies $(x,v)$ with $(\gamma\cdot x, \rho(\gamma)\cdot v)$, $\gamma\in \pi_1(\Sigma)$. Upon associating matrices in $\mathrm{SL}(n,\C)$ with non-degenerate bilinear forms on $\C^n$, an equivariant map to $\mathbb{X}$ is equivalent to a complex bilinear form $Q_f$ on $E_\rho$. The bundle $E_\rho$ comes automatically with a flat connection $D_\rho$ (induced from the trivial connection on $\widetilde{\Sigma}\times \C^n$), which splits uniquely as $$D_\rho=\nabla_f + \psi_f,$$ where $\nabla_f$ is $Q_f$-compatible and $\psi_f$ is $Q_f$-symmetric. We then set $\phi_1^f=\Pi_+\circ \psi_f$ and $\overline{\phi_2}^f=\Pi_-\circ \psi_f$. The standard Maurer-Cartan isomorphism identifies the pullback bundle $f^*T\mathbb{X}$ with $\textrm{End}_0(E)$, and from the uniquely defining properties it is seen that the pullback of the Levi-Civita connection of $T\mathbb{X}$ via $f$ is exactly the connection on $\textrm{End}_0(E)$ induced by $\nabla_f$, and the lift $\widetilde{\psi}_f$ of $\psi_f$ to the universal cover is precisely $\widetilde{\psi}_f=\frac{1}{2}f^{-1}df$. Thus, the complex harmonic map equation is equivalent to $\nabla_f\phi_1^f=\nabla_f\overline{\phi_2}^f=0.$ Putting it all together, $(E_\rho, Q_f, \nabla_f, \phi_1^f, \overline{\phi_2}^f)$ is a complex harmonic bundle; in this construction, the flatness of the connection $D_\rho=\nabla_f + \phi_1^f + \overline{\phi_2}^f$ is manifest.

Conversely, a complex harmonic bundle $(E,Q,\nabla, \phi_1,\overline{\phi_2})$ produces a complex harmonic map: lift everything to the universal cover $\widetilde{\Sigma}$ and pick a point $p\in \widetilde{\Sigma}$. Let $\widetilde{E}$ be the lifted bundle and $\widetilde{Q}$ be the lifted bilinear form. Associating $z\in \widetilde{\Sigma}$ with the parallel transport of $\widetilde{Q}|_{\widetilde{E}_z}$ to $\widetilde{E}_p$ (with respect to $D:=\nabla + \phi_1 +\overline{\phi_2}$) yields a complex harmonic map $\mathbb{X}$, in a way that undoes the construction above. 

From $\widetilde{\psi}_f=\frac{1}{2}f^{-1}df$, we can write the pullback metric as 
\begin{equation}\label{eq: pullback higgs}
    \textrm{tr}(\phi_1\overline{\phi_2})+\frac{1}{2}\textrm{tr}(\phi_1^2)+\frac{1}{2}\textrm{tr}(\overline{\phi_2}^2).
\end{equation}
Hence, the Hopf differentials are expressed $q_1=\frac{1}{2}\textrm{tr}(\phi_1^2)$, $\overline{q_2}=\frac{1}{2}\textrm{tr}(\overline{\phi_2}^2).$

The flatness condition can be expressed via the equation 
\begin{equation}\label{eq: flatness}
    F(\nabla)+[\phi_1,\overline{\phi}_2]=0.
\end{equation}
From this equation, we see that we have a $\C^*$-action on the set of complex harmonic bundles, $$\zeta\cdot (E,Q,\nabla, \phi_1,\overline{\phi_2})=(E,Q,\nabla, \zeta\phi_1,\zeta^{-1}\overline{\phi_2}), \qquad \zeta \in \C^*.$$
Indeed, the holomorphicity conditions are clearly preserved, and from \eqref{eq: flatness}, so is the flatness condition.

Finally, we recall that a $\textrm{SL}(n,\C)$-Higgs bundle on a Riemann surface $(\Sigma,c)$ is a triple $(E,\overline{\partial}_E,\phi)$, where $(E,\overline{\partial}_E)$ is a holomorphic vector bundle and $\phi\in \Omega_{c_1}^{1,0}(\textrm{End}_0(E))$ is holomorphic. We call it a $\textrm{SL}(n,\R)$-Higgs bundle if we also have a non-degenerate bilinear form $Q$ on $E$ such that $\phi$ is $Q$-symmetric. Hitchin's self-duality equations ask for a Hermitian metric $h$ on $E$, called a harmonic metric, such that $\overline{\partial}_E+(\overline{\partial}_E)^{*_h}+\phi+\phi^{*_h}$ is flat. In the $\textrm{SL}(n,\R)$ case, we also impose that $h$ is compatible with $Q$. Then, $(E, Q, \overline{\partial}_E+(\overline{\partial}_E)^{*_h}, \phi, \overline{\phi}^{*_h})$ is a complex harmonic bundle. The complex harmonic map lands in a copy of the real symmetric space $\textrm{SL}(n,\R)/\textrm{SO}(n,\R)$ inside $\mathbb{X}.$

\subsection{Theorem \ref{thm: twistor lines}}\label{sec: proof of twistor line theorem}

Recall from Section \ref{sec: intro, twistor} that a real twistor line for the Deligne-Hitchin moduli space is a family of flat connections of the form 
$$D_\zeta=\overline{\partial}_E+(\overline{\partial}_E)^{*_h}+\zeta\phi +\zeta^{-1}\phi^{*_h}, \qquad \zeta\in \C^*,$$
where $(E,\overline{\partial}_E,\phi)$ is a $\textrm{SL}(n,\C)$-Higgs bundle and $h$ is a harmonic metric. There are excellent expositions on the Deligne-Hitchin moduli space and on twistor lines; we will not need much of this theory here, so we defer the reader to \cite{BeckHellerRoeserEnergy}, \cite{BiswasHellerRoeserReal}, and \cite{SimpsonHodgeFiltration}.

The connection to complex harmonic bundles comes from the $\C^*$-action introduced above. Indeed, a real twistor line is obtained by taking the flat connections along the $\C^*$-orbit of a complex harmonic bundle $(E, Q, \overline{\partial}_E+(\overline{\partial}_E)^{*_h}, \phi, \overline{\phi}^{*_h})$ arising from an $\textrm{SL}(n,\R)$-Higgs bundle. In this way, after integrating out complex harmonic bundles to give complex harmonic maps, we explicitly associate real twistor lines with $\C^*$-families of complex harmonic maps.

Theorem \ref{thm: twistor lines} shows that certain $D_\zeta$'s are partial opers, once we choose the right Riemann surface structure. We do not give a definition here, we just note that a (branched) oper is equivalent to a (branched) complex projective structure (see \cite[Section 3]{silva2024conformal} for detailed explanation). We phrased Theorem \ref{thm: twistor lines} using opers just to make it precise.

We've almost built up enough theory to prove Theorem \ref{thm: twistor lines}, we just need to discuss the hypotheses. Rather than starting with a specific type of Higgs bundle, we have a geometric condition on the harmonic map. A branched immersion between surfaces $f$ is a map that's an immersion on the complement of a discrete set, and which on the discrete set has branch point singularities locally of the form $z\mapsto z^n$. The key fact we'll use below is that the Beltrami form of a branched immersion $\mu_f$ satisfies $|\mu_f|<1$ (at the branch points, $\mu_f=0$). In the case of a harmonic map onto a geodesic, $|\mu_f|=1$ identically. Without either $|\mu_f|<1$ or $|\mu_f|\equiv1$ identically, the proof below will break. 

On a closed surface, the Higgs bundles giving rise to branched immersions and geodesic parameterizations are classified. This is discussed in detail in \cite{DMS}. From \cite{Hitchin:1986vp}, a $\textrm{SL}(2,\R)$-Higgs bundle is equivalent to the data of a line bundle $L$ (an isotropic line for the bilinear form $Q$, so that $E=L\oplus L^{-1}$) and two sections $\alpha\in H^0(\Sigma,L^2\otimes \mathcal{K})$ and $\beta\in H^0(\Sigma, L^{-2}\otimes \mathcal{K})$ ($\alpha$, interpreted as a section of $\textrm{Hom}(L^{-1},L)\otimes\mathcal{K}$, is a component of the Higgs field in the splitting $E=L\oplus L^{-1}$, and similar for $\beta$). The product $\alpha\beta$ is a scalar multiple of the Hopf differential. Both $\alpha$ and $\beta$ come with vanishing divisor $(\alpha)$ and $(\beta)$. From \cite{BBDH} (proof of necessity, most of the sufficiency proof) and \cite{silva2024conformal} (completing the sufficiency proof), the harmonic map is a branched immersion if and only if $(\alpha)\leq (\beta)$ with strict inequality somewhere, or the reverse relation holds (in which case the orientation is swapped). Similarly, the harmonic map parametrizes a geodesic if and only if $(\alpha)=(\beta).$ 

Finally, note that Theorem \ref{thm: twistor lines} does not require the underlying surface to be closed. Instead, we assumed the existence of a harmonic metric. When the surface is closed, one can start with a polystable Higgs bundle, and then by \cite{Hitchin:1986vp} there exists a (usually unique) harmonic metric. About the results on Higgs bundles and branched immersions, analogous results can probably be proved on non-compact surfaces for ``complete" harmonic maps (see \cite{Wan1992ConstantMC}) using generalized maximum principles.

\begin{proof}[Proof of Theorem \ref{thm: twistor lines}]
    One essentially just has to repeat the proof of Theorem \ref{thm: A and B}. Recalling the discussion at the beginning of Section \ref{subsec: Bochner}, for any harmonic immersion $f$ to $\mathbb{H}^2$ we can write the pullback metric as 
   $$ \left(U +\frac 1 U \frac{|q|^2}{ h^2}\right) h +q+\overline{q}.$$ If $f$ is just a branched immersion, then the formula extends over the branch points (we get zero identically). If $f$ parametrizes a geodesic, the same formula holds with $U=\frac{|q|}{h}$. Thus, when we apply the $\C^*$-action, the associated complex harmonic map $f^\zeta=(f_1^\zeta,\overline{f}_z^\zeta)$ has pullback metric 
   \begin{equation}\label{eq: pullback, C star}
       \left(U +\frac 1 U \frac{|q|^2}{ h^2}\right) h +\zeta^2q+\zeta^{-2}\overline{q}.
   \end{equation}
   When we integrate the data of our complex bilinear form using the flat connection, there is an ambiguity not just up to composition with elements of $\textrm{PSL}(2,\C)$, but also the ordering of $f_1^\zeta$ and $\overline{f_2}^\zeta$ is not pre-determined. In the branched immersion case, we specify the initial maps so that $f_1$ is orientation preserving and $\overline{f_2}$ is orientation reversing. In the geodesic case, recall we exclude $\zeta=1$ from the theorem, so instead we specify that $\overline{f_1}^\zeta$ is orientation preserving and $\overline{f_2}^\zeta$ is orientation reversing for some fixed $\zeta$ of norm larger than $1$ and for some fixed $\zeta$ of norm smaller than $1$. From here, Theorem \ref{thm: A and B} goes through verbatim to show that on $A_f$, $f_1^\zeta$ is a branched immersion, and on $B_f$, $\overline{f}_2^\zeta$ is a branched immersion. The only difference is that at zeros of \eqref{eq: pullback, C star}, the complex lines from the proof of Theorem \ref{thm: A and B} degenerate, and hence the maps have branch points.

   Since $D_\zeta$ is the flat connection associated with the $f^\zeta$, the usual correspondence between (branched) projective structures and (branched) opers (see \cite[Section 3]{silva2024conformal}) allows us to say that $D_\zeta$ is a (branched) oper over the relevant Riemann surface.

Holomorphicity of \eqref{eq: pullback, C star} in $\zeta$ is clear. The holomorphicity of the maps $A_f\to \mathcal{CP}(\Sigma)$ and $B_f\to \mathcal{CP}(\overline{\Sigma})$ thus follows from Lemma \ref{lem: holo dependence Bonsante} and the standard fact that the holonomy map on complex projective structures is a local biholomorphism.
\end{proof}

\subsubsection{Relation with work of Gothen-Silva}\label{subsub: gothen silv}
In \cite{silva2024conformal}, Gothen-Silva consider conformal limit rays. That is, fix
$\hbar\in \C^*$, and let $R>0$ range in $(0, |\hbar|^{-1})$. For a polystable Higgs bundle $(E,\overline{\partial}_E,\phi)$ on a closed surface, let $h_R$ be a harmonic metric solving the self-duality equations for $(E,\overline{\partial}_E,R\phi)$. The conformal limit concerns the ray
$$R\mapsto D_{\hbar, R}=\overline{\partial}_E+(\overline{\partial}_E)^{*_{h_R}}+\hbar^{-1}\phi + \hbar R^2\phi^{*_{h_R}}$$ and its behaviour as $R\to 0$. The significance of such rays stems from Gaiotto's conjecture relating the Hitchin section to opers \cite{Gai} (solved in \cite{DFKMMN}, \cite{collier2024conformallimitsparabolicslnchiggs}). Gothen-Silva show that when the hypothesis of Theorem \ref{thm: twistor lines} is satisfied, these rays naturally give a continuous family of complex projective structures. Related to our statement of Theorem \ref{thm: twistor lines}, the condition $|R\hbar|\leq 1$ is saying $\zeta\in A_{f_R}$, where $f_R$ is the harmonic map. We recover their result by applying Theorem \ref{thm: twistor lines} with $\zeta_{\hbar,R}=(\hbar R)^{-1}$ along the ray of Higgs bundles $R\mapsto (E,\overline{\partial}_E,R\phi)$. 

Gothen-Silva prove their main result by manipulating the expressions defining the flat connections and showing that they give (partial) opers for the right choice of complex structure. They then apply the correspondence between (partial) opers and (branched) $\mathbb{CP}^1$-structures. We expect that their method can be pushed through just fine to prove Theorem \ref{thm: twistor lines}. On our end, Theorem \ref{thm: twistor lines} and the manipulations from \cite{silva2024conformal} are given a satisfying geometric interpretation thanks to complex harmonic maps.

\bibliographystyle{plain}
\bibliography{bibliography1}

@article{silva2024conformal,
  title={The conformal limit and projective structures},
  author={Silva, Pedro M. and Gothen, Peter B.},
  journal={International Mathematics Research Notices},
  volume={2024},
  number={16},
  pages={11812--11831},
  year={2024},
  publisher={Oxford University Press},
  doi={10.1093/imrn/rnae142},
  url={https://doi.org/10.1093/imrn/rnae142}
}

@article{KrasnovSchlenker2008,
  author    = {Krasnov, Kirill and Schlenker, Jean-Marc},
  title     = {On the Renormalized Volume of Hyperbolic 3-Manifolds},
  journal   = {Communications in Mathematical Physics},
  volume    = {279},
  number    = {3},
  pages     = {637--668},
  year      = {2008},
  doi       = {10.1007/s00220-008-0423-7},
  url       = {https://doi.org}
}

@incollection{SimpsonHodgeFiltration,
  author    = {Simpson, Carlos},
  title     = {The Hodge filtration on nonabelian cohomology},
  booktitle = {Algebraic Geometry---Santa Cruz 1995},
  series    = {Proceedings of Symposia in Pure Mathematics},
  volume    = {62},
  part      = {2},
  pages     = {217--281},
  publisher = {American Mathematical Society},
  address   = {Providence, RI},
  year      = {1997},
  doi       = {10.1090/pspum/062.2/1492538},
}

@article{BBDH,
author = {Biswas, Indranil and Bradlow, Steven and Dumitrescu, Sorin and Heller, Sebastian},
title = {Uniformization of branched surfaces and Higgs bundles},
journal = {International Journal of Mathematics},
volume = {32},
number = {13},
pages = {2150096},
year = {2021},
doi = {10.1142/S0129167X21500968},
}

@article{Gai,
    author = "Gaiotto, D.",
    title = "{Opers and TBA}",
    eprint = "1403.6137",
    archivePrefix = "arXiv",
    primaryClass = "hep-th",
    month = "3",
    year = "2014"
}

@misc{collier2024conformallimitsparabolicslnchiggs,
      title={Conformal limits for parabolic SL(n,C)-Higgs bundles}, 
      author={Brian Collier and Laura Fredrickson and Richard Wentworth},
      year={2024},
      eprint={2407.16798},
      archivePrefix={arXiv},
      primaryClass={math.DG},
      url={https://arxiv.org/abs/2407.16798}, 
}

@misc{DMS,
      title={Domination between non-Fuchsian representations and anti-de Sitter geometry}, 
      author={Farid Diaf and Abderrahim Mesbah and Nathaniel Sagman},
      year={2025},
      eprint={2511.10570},
      archivePrefix={arXiv},
      primaryClass={math.DG},
      url={https://arxiv.org/abs/2511.10570}, 
}

@article{ElE,
  title={A metric uniformization model for the Quasi-Fuchsian space},
  author={El Emam, Christian},
  journal={Transactions of the American Mathematical Society},
  volume={379},
  number={04},
  pages={2569--2604},
  year={2026}
}

@article{Hitchin:1986vp,
    author = "Hitchin, Nigel J.",
    title = "{The Self duality equations on a Riemann surface}",
    doi = "10.1112/plms/s3-55.1.59",
    journal = "Proc. Lond. Math. Soc.",
    volume = "55",
    pages = "59--131",
    year = "1987"
}

@article{SchoenYauUnivalent,
author = {Schoen, Richard and Yau, Shing-Tung},
title = {On Univalent Harmonic Maps Between Surfaces},
journal = {Inventiones Mathematicae},
volume = {44},
number = {3},
pages = {265--278},
year = {1978},
doi = {10.1007/BF01403164}
}

@article {Li,
    AUTHOR = {Li, Qiongling},
     TITLE = {An introduction to {H}iggs bundles via harmonic maps},
   JOURNAL = {SIGMA Symmetry Integrability Geom. Methods Appl.},
  FJOURNAL = {SIGMA. Symmetry, Integrability and Geometry. Methods and
              Applications},
    VOLUME = {15},
      YEAR = {2019},
     PAGES = {Paper No. 035, 30},
      ISSN = {1815-0659},
   MRCLASS = {53C43 (53C07 53C21)},
  MRNUMBER = {3947036},
MRREVIEWER = {Marco\ Castrillon Lopez},
       DOI = {10.3842/SIGMA.2019.035},
       URL = {https://doi.org/10.3842/SIGMA.2019.035},
}

@article {BEE,
    AUTHOR = {Bonsante, Francesco and El Emam, Christian},
     TITLE = {On immersions of surfaces into {$SL(2,\Bbb C)$} and geometric
              consequences},
   JOURNAL = {Int. Math. Res. Not. IMRN},
  FJOURNAL = {International Mathematics Research Notices. IMRN},
      YEAR = {2022},
    NUMBER = {12},
     PAGES = {8803--8864},
      ISSN = {1073-7928,1687-0247},
   MRCLASS = {53C42 (32Q30 53C21 57K31)},
  MRNUMBER = {4436196},
MRREVIEWER = {Erasmo\ Caponio},
       DOI = {10.1093/imrn/rnab189},
       URL = {https://doi.org/10.1093/imrn/rnab189},
}

@article {DFKMMN,
    AUTHOR = {Dumitrescu, Olivia and Fredrickson, Laura and Kydonakis,
              Georgios and Mazzeo, Rafe and Mulase, Motohico and Neitzke,
              Andrew},
     TITLE = {From the {H}itchin section to opers through nonabelian
              {H}odge},
   JOURNAL = {J. Differential Geom.},
  FJOURNAL = {Journal of Differential Geometry},
    VOLUME = {117},
      YEAR = {2021},
    NUMBER = {2},
     PAGES = {223--253},
      ISSN = {0022-040X,1945-743X},
   MRCLASS = {53C07 (14D21 58E15 81T13)},
  MRNUMBER = {4214341},
MRREVIEWER = {Markus\ R\"{o}ser},
       DOI = {10.4310/jdg/1612975016},
       URL = {https://doi.org/10.4310/jdg/1612975016},
}

@article {D,
    AUTHOR = {Donaldson, S. K.},
     TITLE = {Twisted harmonic maps and the self-duality equations},
   JOURNAL = {Proc. London Math. Soc. (3)},
  FJOURNAL = {Proceedings of the London Mathematical Society. Third Series},
    VOLUME = {55},
      YEAR = {1987},
    NUMBER = {1},
     PAGES = {127--131},
      ISSN = {0024-6115},
   MRCLASS = {58E20 (32L15 53C05)},
  MRNUMBER = {887285},
MRREVIEWER = {Mitsuhiro Itoh},
       DOI = {10.1112/plms/s3-55.1.127},
       URL = {https://doi.org/10.1112/plms/s3-55.1.127},
}

@article {C,
      title={Finite order automorphisms of Higgs bundles: theory and application}, 
      author={Brian Collier},
   JOURNAL = {PhD thesis},
    YEAR={2016}
}

@article{Kuwert1996,
  author  = {Kuwert, Ernst},
  title   = {Harmonic maps between flat surfaces with conical singularities},
  journal = {Mathematische Zeitschrift},
  volume  = {221},
  number  = {3},
  pages   = {421--436},
  year    = {1996},
  mrnumber = {1381590},
  zbl      = {0871.58028}
}

@article{PanWolf2026,
  author  = {Pan, Huiping and Wolf, Michael},
  title   = {Ray structures on {Teichm{\"u}ller} space},
  journal = {Acta Mathematica},
  volume  = {236},
  number  = {1},
  pages   = {51--208},
  year    = {2026},
  doi     = {10.4310/ACTA.2026.v236.n1.a2},
}

@article {SU,
    AUTHOR = {Bers, Lipman},
     TITLE = {Simultaneous uniformization},
   JOURNAL = {Bull. Amer. Math. Soc.},
  FJOURNAL = {Bulletin of the American Mathematical Society},
    VOLUME = {66},
      YEAR = {1960},
     PAGES = {94--97},
      ISSN = {0002-9904},
   MRCLASS = {30.00},
  MRNUMBER = {111834},
MRREVIEWER = {H.\ L.\ Royden},
       DOI = {10.1090/S0002-9904-1960-10413-2},
       URL = {https://doi.org/10.1090/S0002-9904-1960-10413-2},
}

@incollection {Uh,
    AUTHOR = {Uhlenbeck, Karen K.},
     TITLE = {Closed minimal surfaces in hyperbolic {$3$}-manifolds},
 BOOKTITLE = {Seminar on minimal submanifolds},
    SERIES = {Ann. of Math. Stud.},
    VOLUME = {103},
     PAGES = {147--168},
 PUBLISHER = {Princeton Univ. Press, Princeton, NJ},
      YEAR = {1983},
      ISBN = {0-691-08324-X; 0-691-08319-3},
   MRCLASS = {53C42 (32G15 53A10)},
  MRNUMBER = {795233},
MRREVIEWER = {R.\ Osserman},
}

@article {W2,
    AUTHOR = {Wolf, Michael},
     TITLE = {The {T}eichm\"{u}ller theory of harmonic maps},
   JOURNAL = {J. Differential Geom.},
  FJOURNAL = {Journal of Differential Geometry},
    VOLUME = {29},
      YEAR = {1989},
    NUMBER = {2},
     PAGES = {449--479},
      ISSN = {0022-040X},
   MRCLASS = {58E20 (32G15 58D17)},
  MRNUMBER = {982185},
MRREVIEWER = {A. J. Tromba},
       URL = {http://projecteuclid.org/euclid.jdg/1214442885},
}

@article{BeckHellerRoeserEnergy,
  author  = {Beck, Florian and Heller, Sebastian and R{\"o}ser, Markus},
  title   = {Energy of sections of the {Deligne--Hitchin} twistor space},
  journal = {Mathematische Annalen},
  volume  = {380},
  pages   = {1169--1214},
  year    = {2021},
  doi     = {10.1007/s00208-020-02042-0},
}

@article{WanAu1994,
author = {Wan, Tom Y.-H. and Au, T. K.-K.},
title = {Parabolic Constant Mean Curvature Spacelike Surfaces},
journal = {Proceedings of the American Mathematical Society},
volume = {120},
number = {2},
pages = {559--564},
year = {1994}
}

@article{Wan1992ConstantMC,
  author    = {Wan, Tom Yau-Heng},
  title     = {Constant mean curvature surface, harmonic maps, and universal {T}eichm{\"u}ller space},
  journal   = {Journal of Differential Geometry},
  year      = {1992},
  volume    = {35},
  number    = {3},
  pages     = {643--657},
  publisher = {Lehigh University},
  url       = {https://projecteuclid.org}
}

@article{LiVortexEquations,
author = {Li, Qiongling},
title = {On the Uniqueness of Vortex Equations and Its Geometric Applications},
journal = {The Journal of Geometric Analysis},
volume = {29},
number = {1},
pages = {105--120},
year = {2019},
doi = {10.1007/s12220-018-9981-x}
}

@article{SchulzNilpotent,
  author  = {Schulz, Sebastian},
  title   = {Nilpotent Higgs bundles and families of flat connections},
  journal = {Communications in Mathematical Physics},
  volume  = {403},
  number  = {2},
  pages   = {877--915},
  year    = {2023},
  doi     = {10.1007/s00220-023-04811-9},
}

@article{BiswasHellerRoeserReal,
  author  = {Biswas, Indranil and Heller, Sebastian and R{\"o}ser, Markus},
  title   = {Real holomorphic sections of the {Deligne--Hitchin} twistor space},
  journal = {Communications in Mathematical Physics},
  volume  = {366},
  number  = {3},
  pages   = {1099--1133},
  year    = {2019},
  doi     = {10.1007/s00220-019-03340-8},
}

@book {GuillPoll,
    AUTHOR = {Guillemin, Victor and Pollack, Alan},
     TITLE = {Differential topology},
      NOTE = {Reprint of the 1974 original},
 PUBLISHER = {AMS Chelsea Publishing, Providence, RI},
      YEAR = {2010},
     PAGES = {xviii+224},
      ISBN = {978-0-8218-5193-7},
   MRCLASS = {58-01 (57-01)},
  MRNUMBER = {2680546},
       DOI = {10.1090/chel/370},
       URL = {https://doi.org/10.1090/chel/370},
}

@misc{ElSa,
      title={Complex affine spheres and a Bers theorem for SL(3,C)}, 
      author={Christian El Emam and Nathaniel Sagman},
      year={2025},
      eprint={2406.15287},
      archivePrefix={arXiv},
      primaryClass={math.DG},
      url={https://arxiv.org/abs/2406.15287}, 
}

@article{ESholodependence,
  title={Holomorphic dependence for the Beltrami equation in Sobolev spaces},
  author={El Emam, Christian and Sagman, Nathaniel},
  journal={Proceedings of the American Mathematical Society},
  volume={154},
  number={5},
  pages={1973--1989},
  year={2026},
  doi={10.1090/proc/17402},
  url={https://doi.org}
}

@misc{ES_Chmaps,
      title={Complex harmonic maps and rank 2 higher Teichm\"uller theory}, 
      author={Christian El Emam and Nathaniel Sagman},
      year={2025},
      eprint={2506.11746},
      archivePrefix={arXiv},
      primaryClass={math.DG},
      url={https://arxiv.org/abs/2506.11746}, 
}

\end{document}